\documentclass[pdflatex,12pt,a4paper,twoside]{article}
\usepackage[a4paper,left=30mm,right=29mm,top=27mm,bottom=27mm]{geometry}
\usepackage{tikz}

\usepackage{amsmath,amsfonts,euscript,amssymb,amsthm}
\usepackage{graphicx, epsfig}
\usepackage{stackengine}
\usepackage[update,prepend]{epstopdf}

\newcommand{\R}{\mathbb{R}}
\newcommand{\N}{\mathbb{N}}
\newcommand{\C}{\mathbb{C}}

\newcommand{\eps}{\varepsilon}
\def\ri{{\rm i}}

\newcommand{\vth}{\vartheta}

\newcommand{\mean}{-\hspace{-1.15em}\int}
\newcommand{\textmean}{-\hspace{-0.85em}\int}

\renewcommand{\Re}{\mathrm{Re}}
\renewcommand{\Im}{\mathrm{Im}}

\newcommand{\spa}{\mathrm{span}}
\newcommand{\loc}{\mathrm{loc}}

\newcommand{\one}{\mathrm{\bf 1}}

\def\calL{\mathcal{L}}

\def\Xint#1{\mathchoice
   {\XXint\displaystyle\textstyle{#1}}%
   {\XXint\textstyle\scriptstyle{#1}}%
   {\XXint\scriptstyle\scriptscriptstyle{#1}}%
   {\XXint\scriptscriptstyle\scriptscriptstyle{#1}}%
   \!\int}
\def\XXint#1#2#3{{\setbox0=\hbox{$#1{#2#3}{\int}$}
     \vcenter{\hbox{$#2#3$}}\kern-.5\wd0}}
\def\meanint{\Xint-}

\newtheorem{theorem}{Theorem}[section]
\newtheorem{definition}[theorem]{Definition}
\newtheorem{lemma}[theorem]{Lemma}
\newtheorem{proposition}[theorem]{Proposition}

\newtheorem{assumption}[theorem]{Assumption}

\newtheorem{remark}[theorem]{Remark}

\numberwithin{equation}{section}

\begin{document}

\thispagestyle{empty}
\begin{center}
  \vspace*{7mm}
  {\Large\bf A Bloch wave numerical scheme for scattering\\[3mm]
    problems in periodic wave-guides}

  \vspace*{5mm} {\large Tom\'a\v{s} Dohnal and Ben
    Schweizer\footnote{Technische Universit\"at Dortmund, Fakult\"at
      f\"ur Mathematik, Vogelpothsweg 87, D-44227 Dortmund,
      Germany. {\tt tomas.dohnal@tu-dortmund.de} and {\tt
        ben.schweizer@tu-dortmund.de} }}
  
  \vspace*{4mm}
  August 1, 2017
\end{center}

\vspace*{2mm}
\begin{center}
   \begin{minipage}[c]{0.86\textwidth}
     {\bf Abstract:} We present a new numerical scheme to solve the
     Helmholtz equation in a wave-guide. We consider a medium that is
     bounded in the $x_2$-direction, unbounded in the $x_1$-direction
     and $\eps$-periodic for large $|x_1|$, allowing different media
     on the left and on the right. We suggest a new numerical method
     that is based on a truncation of the domain and the use of Bloch
     wave ansatz functions in radiation boxes.  We prove the existence
     and a stability estimate for the infinite dimensional version of
     the proposed problem. The scheme is tested on several interfaces
     of homogeneous and periodic media and it is used to investigate
     the effect of negative refraction at the interface of a photonic
     crystal with a positive effective refractive index.

     \vspace*{2mm} {\bf Key-words:} Helmholtz equation, radiation
     condition, locally periodic media, Bloch waves, negative refraction

     \vspace*{2mm}
   \end{minipage}
\end{center}

\vspace*{3mm}

\section{Introduction}

Acoustic, elastic, and electromagnetic waves are of quite different
nature, but in many geometries of practical relevance, they can all be
described by the linear wave equation of second order. If we are
interested in the distribution of the wave intensity in a domain
$\Omega\subset \R^n$ after a transitional time, we have to solve the
time harmonic problem, which is the Helmholtz equation
\nocite{Helmholtz-1860}
\begin{equation}
  \label{eq:Waveguide}
  -\nabla\cdot (a \nabla u) = \omega^2 u + f\qquad\text{ in }\Omega\,.
\end{equation}
In this equation, $f: \Omega\to \R$ is a source, the positive
coefficient field $a: \Omega\to \R$ describes the properties of the
medium, we assume that $a$ equals an $\eps$-periodic function $a_+(x)$
on the right of a bounded region and an $\eps$-periodic function
$a_-(x)$ on the left. The periodicity $\eps>0$ and the frequency
$\omega>0$ are given, we consider them as fixed parameters throughout
this work. The aim is to find the solution $u: \Omega \to \C$ to
equation \eqref {eq:Waveguide}.

We are interested in the analysis of wave-guides. Restricting to the
two-dimen\-sional case for simplicity, we consider the infinite strip
$\Omega := \R\times (0,H)$ with the height $H>0$.  It remains to
choose boundary conditions. The main difficulty comes from the
radiation conditions that have to be imposed for $x_1\to \pm\infty$.
In contrast, the analysis is essentially independent of the boundary
condition on the lateral boundary
$(\R\times \{ 0 \}) \cup (\R\times \{ H \})$. To make a choice, we
work with periodicity conditions as in \cite {Lamacz-Schweizer-Bloch}:
values and derivatives coincide at $(x_1,0)\in \R\times \{ 0 \}$ and
$(x_1,H)\in \R\times \{ H \}$.

{\bf Radiation condition.}  At the lateral boundaries, for
$x_1\to \pm\infty$, we have to impose a radiation condition. It is not
an easy task to formulate the radiation condition in a wave-guide.
For a periodic semi-infinite wave-guide, a condition was formulated
and analyzed in \cite {Hoang2011}, similarly, the periodic wave-guide
was analyzed in \cite {FlissJoly2016, JolyLiFliss2006}, a wave-guide
with different coefficients in the two infinite directions was
analyzed in \cite {Lamacz-Schweizer-Bloch}.  The latter publication
provides a the uniqueness result for the suggested radiation
condition. It is this radiation condition which we base our numerical
scheme on. Let us formulate the following fact in order to illustrate
the complexity of the wave-guide problem \eqref {eq:Waveguide}. Let
the coefficients $a$ be real and bounded from below by a positive
number $a_0$.  Furthermore, assume that $a$ is $\eps$-periodic on
$x_1<0$ and on $x_1>0$. Let the source $f:\Omega\to \R$ of class
$L^2(\Omega)$ have a bounded support.  Open question: Does \eqref
{eq:Waveguide} possess a radiating solution $u$?

In the above question we avoided the precise
formulation of the radiation condition -- this is adequate since the
different forms of radiation conditions in a periodic wave-guide are
essentially equivalent.  The form of the radiation condition that was
suggested in \cite {Lamacz-Schweizer-Bloch} can be written as
\begin{equation}
  \label{eq:outgoingright-intro}
  \mean_{RY_\eps} \left|\Pi^+_{<0}(\left\{ u\right\}^+_{R,R})
  \right|^2\to 0\quad  \text{ as }\quad  R\to \infty\,.
\end{equation}
The condition uses the periodicity cell $Y_\eps = (0,\eps)^2$, the
function $\left\{ u\right\}^+_{R,L}$, which is the restriction of $u$
(periodically extended in the vertical direction) to the domain
$\eps (R ,R+L)\times \eps (0,R)$. For large $R$, we hence consider the
solution $u$ on the far right. Note that we use here boxes of width
$\eps L$ at position $\eps R$, while only $L = R$ was considered in
\cite {Lamacz-Schweizer-Bloch}.  The symbol $\Pi^+_{<0}$ denotes the
projection of the argument (which is a function on a square in $\R^2$)
onto the space spanned by left-going Bloch waves, i.e.\,Bloch waves
for which the first component of the Poynting vector is negative.  The
superscript ``$+$'' indicates that the Bloch waves are calculated for
the periodic coefficient $a$ of the right half-cylinder. The symbol
$\mean$ denotes the mean value integral,
$\mean_A f := |A|^{-1} \int_A f$.  A condition analogous to \eqref
{eq:outgoingright-intro} with a projection onto right-going waves must
be imposed on the left.
 
\medskip The aim of this contribution is to introduce and to analyze a
numerical scheme that can be used to solve the wave-guide problem.

\subsection{An approach based on \eqref {eq:outgoingright-intro}}

The numerical problem must be formulated in a bounded
domain. Furthermore, as in many other related approaches, we must
additionally introduce a absorption coefficient $\delta\ge 0$. Our
analytical results will only cover the case $\delta>0$, but practical
experience shows that the numerical scheme works well also for
$\delta=0$.

The truncation is performed with two positive integer parameters
$R,L\in \N$.  We use the inner domain
$\Omega_R := (-R \eps, R\eps) \times (0,H)_\sharp$ with $H=\eps K$,
$K\in \N$, and the extended domain
$\Omega_{R+L} := (-(R+L) \eps, (R+L) \eps) \times (0,H)_\sharp$, the
symbol $\sharp$ indicates that we demand periodicity conditions in
vertical direction.  In the following we suppress the dependencies on
$\delta>0$ and $L>0$, and denote the unknown function on the truncated
domain as $u = u_{R}: \Omega_{R+L} \to \C$. We impose that $u$
satisfies the Helmholtz equation on the inner domain with the
absorption parameter $\delta$:
\begin{equation}
  \label{eq:Waveguide-R-L}
  -\nabla\cdot (a \nabla u) = \omega^2 (1+\ri\delta) u + f
  \qquad\text{ in }\Omega_R\,.
\end{equation}
In order to formulate the radiation condition, we use the positive
parameter $L>0$ and consider the radiation boxes
$W_{R,L}^+ := \eps (R, R+L) \times (0,H)_\sharp$ and
$W_{R,L}^- := \eps (-(R+L), -R) \times (0,H)_\sharp$ as sketched in
Figure \ref {fig:geometry}.
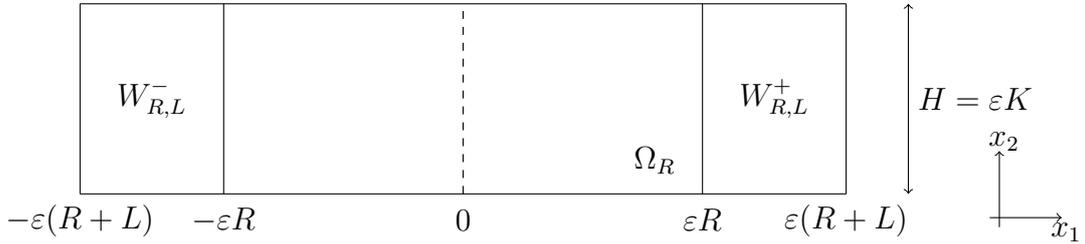
\begin{figure}[ht]
  \centering
  \begin{tikzpicture}[scale = 0.63]
    \draw[-] (-8,0)--(8,0);
    \draw[-] (-8,4)--(8,4);
    \draw[-] (-8,0)--(-8,4);
    \draw[-] (-5,0)--(-5,4);
    \draw[-] (8,0)--(8,4);
    \draw[-] (5,0)--(5,4);

    \draw[dashed] (0,0)--(0,4);


    \node[] at (8,-.55) {$\eps (R+L)$};
    \node[] at (5,-.55) {$\eps R$};
    \node[] at (-8,-.55) {$-\eps (R+L)$};
    \node[] at (-5,-.55) {$-\eps R$};
    \node[] at (0, -.55) {$0$};

    \draw[<->] (9.3,0)--(9.3,4);
    \node[] at (10.7, 2) {$H = \eps K$};

    \node[] at (6.5,2) {$W_{R,L}^+$};
    \node[] at (-6.5,2) {$W_{R,L}^-$};
    \node[] at (4,0.7) {$\Omega_{R}$};

    \draw[->] (11,-0.5)--(12.5,-0.5);
    \node[] at (12.6, -.8) {$x_1$};

    \draw[->] (11.2,-0.7)--(11.2,0.9);
    \node[] at (11.3, 1.1) {$x_2$};
\end{tikzpicture}
  \caption{Geometry of the truncated domain}
  \label{fig:geometry}
\end{figure}
The restrictions of a function $u:\Omega_{R+L}\to \C$ to these two
rectangles are denoted as $\left\{ u\right\}^+_{R,L}$ and $\left\{
  u\right\}^-_{R,L}$. More precisely, we additionally shift the lower
left corner to the origin and set, for $x_1\in [0, \eps L)$ and
$x_2\in [0, H)$,
\begin{equation*}
  \left\{ u\right\}^+_{R,L}(x_1, x_2) := u(\eps R + x_1, x_2)\,,\quad
  \left\{ u\right\}^-_{R,L}(x_1, x_2) := u(-\eps (R+L) + x_1, x_2)\,.
\end{equation*}
We emphasize that the radiation boxes have width $\eps L$ and are
positioned at $\pm \eps R$, while we restricted ourselves to $L = R$ in
\cite {Lamacz-Schweizer-Bloch} to simplify notations.  The idea is now
to impose \eqref {eq:outgoingright-intro} and its counterpart on the
left hand side in a strong form at a finite distance; in this first
attempt we demand
\begin{equation}
  \label{eq:outgoing-R-L}
  \Pi^+_{<0}(\left\{ u\right\}^+_{R,L}) = 0\qquad\text{and}\qquad
  \Pi^-_{>0}(\left\{ u\right\}^-_{R,L}) = 0\,.
\end{equation}
The projections are defined below in Definition \ref {def:projection}.
Condition \eqref {eq:outgoing-R-L} expresses that, in the right
radiation box $W_{R,L}^+$, the solution does not contain left-going
waves, and, in the left radiation box $W_{R,L}^-$, the solution does
not contain right-going waves.

\paragraph{Coupling conditions across interfaces.}

We finally have to demand conditions along the two interior interfaces
$\Gamma_R^+ := \overline{\Omega_R} \cap \overline{W_{R,L}^+} = \{\eps
R\} \times (0,H)_\sharp$ and
$\Gamma_R^- := \overline{\Omega_R} \cap \overline{W_{R,L}^-} = \{-\eps
R\} \times (0,H)_\sharp$. We impose on $u$ the weak continuity
condition
\begin{equation}
  u \in H^1(\Omega_{R+L})\,.\label{eq:weak-cont}
\end{equation}

A second condition is needed to replace the continuity of the flux,
which may be expressed as
\begin{equation}
  \label{eq:flux-R-L}
  \left[ e_1\cdot a\nabla u \right]_{\Gamma_R^\pm} = 0\,,
\end{equation}
where the bracket $[ . ]_{\Gamma_R^\pm}$ denotes the jump of a
function across the interface $\Gamma_R^\pm$.

Let us assume that the problem parameters $a$, $\omega$, $\delta$, and
the truncation parameters $R$ and $L$ are fixed. Our first attempt to
define a truncated problem is the following.
\begin{quote}
  (P$_0$): Given $f$, find $u$ that satisfies \eqref
  {eq:Waveguide-R-L}--\eqref {eq:flux-R-L}.
\end{quote} {\em Warning:} Problem (P$_0$) is not a useful truncated
problem since it does not contain a partial differential equation in
the boxes $W_{R,L}^\pm$.

The main result of this work is the formulation of a more useful
problem (P).  Problem (P) will be defined with a function space $V$,
which strengthens \eqref{eq:outgoing-R-L} and with a bilinear form
$\beta$, which encodes a weaker version of \eqref{eq:flux-R-L}, see
\eqref {eq:beta-problem}. Theorem \ref {thm:existence} provides the
solvability of this problem and a stability estimate.  The numerical
results presented in Section \ref {sec.numerics} are obtained using
problem (P).

\subsection{Literature}

An outgoing wave condition for homogeneous media was suggested by
Sommerfeld in 1912, today it is the undoubted radiation condition for
the full space problem. If, for numerical purpose, the domain is
truncated, the radiation condition must be replaced by a condition at
a finite distance. One of the ideas is to use a boundary condition that
exploits a representation of the solution outside the truncated domain
(integral representation or Dirichlet-to-Neumann map). Another idea is
to introduce an absorbing layer that surrounds the truncated domain
(perfectly matched layer technique).

\paragraph{Radiation conditions in periodic media.}

The two sketched ideas cannot easily be adapted to treat periodic
media: Integral representations are not available and a non-reflecting
boundary condition is not exact, since a non-homogeneous medium always
reflects waves in part. The derivation of perfectly matched layers
typically requires an explicit representation of propagating modes,
which is not available in periodic media.

Outgoing wave conditions in periodic wave-guides have been introduced
and analyzed e.g.\,in \cite {FlissJoly2016, Hoang2011}. Loosely
speaking, a radiating solution is a function that consists, at large
distances from the origin and up to small errors, of outgoing Bloch
waves. The two contributions \cite {FlissJoly2016, Hoang2011} treat
the (globally) periodic wave-guide problem and the periodic
half-wave-guide problem, respectively, and they contain existence
results that are based on a limiting absorption principle. A slightly
different radiation condition for the locally periodic wave-guide was
suggested in \cite {Lamacz-Schweizer-Bloch}.  Regarding further
results on limiting absorption principles we mention \cite
{JolyLiFliss2006, Radosz2015}.

Other conditions are the ``modal radiation condition'', formulated in
Definition 2.4 of \cite{Bonnet-Ben-etal-SIAP2009} and the ``pole
condition'' of \cite {Hohage-Schmidt-Z-2003}.  We mention \cite
{Nazarov2014} and the references therein for other approaches to
radiation conditions.

Regarding the general treatment of waves in periodic media (e.g.\,in
photonic crystals) we refer to \cite {PhotonicCrystals-book,
  Kuchment-2001}. Regarding the tool of Bloch expansions and Bloch
measures, we refer to \cite{Allaire-Conca-1998}.

\paragraph{Numerical treatment of radiation conditions.}

The numerical treatment of exact boundary conditions in an
inhomogeneous material was considered in \cite{Fliss-Joly-2009},
extending the approach of \cite {JolyLiFliss2006} to a material that
is inhomogeneous in two directions. Their approach uses
Dirichlet-to-Neumann maps that are defined by half-infinite wave guide
problems.  The authors provide an explanation why none of the
classical approaches to implement outgoing wave conditions at finite
distance (local radiation condition, perfectly matched layers,
standard Dirichlet-to-Neumann maps) can easily be adapted to periodic
media. We note that, just as in the contribution at hand, the analysis
of \cite{Fliss-Joly-2009} is restricted to the case with positive
losses.

The method was developed further to a numerical scheme in
\cite{Joly-Fliss-2012}. The ideas were used in \cite{Fliss-2013} to
the study of line defects in a periodic photonic crystal; these
defects have been analyzed also in \cite{HoangRadosz-2014} with the
result that a line defect cannot support finite energy modes (bound
states).  For extensions of the numerical scheme to Robin type
boundary conditions see \cite {FlissKlindworthSchmidt2015}.

\vspace*{3mm} {\em Enriched finite elements.} Our numerical method is
a Galerkin method in which we use two different types of ansatz
functions: Standard piecewise linear hat functions in the interior of
the domain and Bloch waves in the radiation boxes. The approach is
reminiscent of enriched finite element methods, see
e.g.\,\cite{KOH2011, SBH2006}.

\paragraph{Negative refraction.}

Photonic crystals can exhibit astonishing behavior -- one of them is
negative refraction. When a planar wave hits the interface between
free space and photonic crystal, then one part of the wave is
reflected, another part generates waves inside the photonic
crystal. It has been observed that the waves in the crystal can travel
in a direction that corresponds to negative refraction.

There are two explanations for this effect. The first one is based on
a study of a homogenized equivalent medium that replaces the photonic
crystal. This replacement provides a good approximation if the
periodicity is small compared to the wave-length. We refer to Figure
\ref {F:comp_hom_om_sm} for the numerical results of our method in
this case. Indeed, the results show a good agreement between the
solution with the periodic medium on the right and the solution with
the homogenized medium. If the homogenized medium happens to have a
negative index of refraction, then negative refraction is visible in
the homogenized problem and also in the periodic problem (not the case
in Figure \ref {F:comp_hom_om_sm}).  For the underlying idea we
mention \cite{Pendry2000}, for mathematical justifications we refer to
\cite{BFe2, BouchitteSchweizer-Max, Lamacz-Schweizer-Max,
  Lamacz-Schweizer-Neg}.  In \cite {EfrosPokrovsky-SolidState-2004,
  Pokrovsky2003333} the negative refraction effect is explained in the
spirit of negative index materials.

The second explanation of the effect of negative refraction, observed
and outlined in \cite {PhysRevB.65.201104}. The main point of \cite
{PhysRevB.65.201104} is that negative refraction can occur between two
materials with positive index (where no negative refraction occurs in
the homogenization limit). The analysis is purely based on the study
of the band structure of the left and right medium.  Figure \ref
{F:comp_hom_om_lg} illustrates this effect: The incoming wave from the
left travels north-east. For the homogenized material in the right
half (results of the bottom figure), the transmitted wave also travels
north-east. In contrast, for wave-length and periodicity of comparable
size (top figure), the transmitted wave travels south-east. Both \cite
{Lamacz-Schweizer-Bloch} and the work at hand support the
interpretation of \cite {PhysRevB.65.201104}: Negative refraction is
possible in positive index materials.  We emphasize that we use here
the same photonic crystal that was also used in \cite
{EfrosPokrovsky-SolidState-2004} and \cite {PhysRevB.65.201104}. This
periodic medium does not have a negative effective index in the sense
of homogenization.

\section{Bloch expansion formalism and problem (P)}\label{S:Bloch}

We have to fix the notations of the Bloch expansion formalism. The
formalism allows us, on the one hand, to define the projections that
have already been used in condition \eqref {eq:outgoing-R-L}. On the
other hand, we will be able to formulate the modified problem
(P), which we suggest as a useful truncated problem.

We assume in the following that the medium is $\eps$-periodic on the
right and on the left. More precisely, for two $Y_\eps$-periodic
functions $a_+$ and $a_-$, we assume $a(x) = a_+(x)$ for
$x_1 \ge \eps R/2$ and $a(x) = a_-(x)$ for $x_1 \le - \eps R/2$. We
work in two space dimensions, but the methods are not restricted to
this case.

\subsection {Bloch formalism}

For $\eps>0$ let $Y_\eps = \eps (0,1)^2$ be the periodicity cell and
let $H = \eps K$ with $K\in \N$ be the height of the domain $\Omega =
\R\times (0,H)_\sharp$.  We use the finite index set $Q_{K} :=
\{0,\frac{1}{K}, \frac{2}{K}, \dots, \frac{K-1}{K}\}$ and employ a
Pre-Bloch expansion in the vertical direction: Any function $u\in
L^2_\loc(\R\times (0,H);\C)$ can be expanded in periodic functions
with phase-shifts: There is a unique family of $\eps$-periodic
functions $\Phi_{j_2}(x_1, \cdot)$ such that, in the sense of
$L^2_\loc( \R\times (0,H); \C)$,
\begin{equation}
  \label{eq:discrete-u-1}
  u(x_1, x_2) 
  = \sum_{j_2\in Q_{K}} \Phi_{j_2}(x_1, x_2)\, e^{2\pi \ri j_2 x_2/\eps}\,.
\end{equation}
The analogous result holds when we expand a function $u\in L^2((0,\eps
L)\times (0, \eps K); \C)$ in both directions $x_1$ and $x_2$. In this
case one obtains for $j= (j_1, j_2)\in Q_L\times Q_K$ functions
$\Phi_j = \Phi_j(x_1, x_2)$ that are $\eps$-periodic in both
directions.

We regard the $\eps$-periodic functions $\Phi_{j}$ for
$j = (j_1, j_2)$ as maps $Y_\eps \to \C$ and expand them in terms of
eigenfunctions of the operator
\begin{equation}\label{eq:L_j-operator}
  \calL_j^\pm
  := -\left(\nabla + 2\pi \ri j/\eps\right) 
  \cdot \left(a_\pm(x) \left(\nabla+ 2\pi \ri j/\eps\right)\right)\,,
\end{equation}
which is defined on $H^1_\sharp(Y_\eps; \C)$.  The definition of the
operator $\calL_j^\pm$ is motivated by the following fact: If
$\Psi_{j}^\pm$ is an eigenfunction of $\calL_j^\pm$ with eigenvalue
$\mu_j^\pm$, then $\Psi_{j}^\pm e^{2\pi \ri j\cdot x / \eps}$ is a
solution to the Helmholtz equation \eqref {eq:Waveguide} with
$a = a_\pm$ and $\omega^2 = \mu_j^\pm$ and $f=0$.

\begin{definition}[Bloch eigenfunctions]
  \label{def:Bloch-eigenfunctions}
  For $\eps>0$ and $j\in [0,1]^2$ we denote by $\left(\Psi^\pm_{j,m}
  \right)_{m\in\N_0}$ an orthogonal family of eigenfunctions to the
  symmetric operator $\calL_j^\pm$, ordered to have
  $\mu_{m+1}^\pm(j)\geq \mu_m^\pm(j)$ for all $m\in\N_0$. We normalize
  with $\textmean_{Y_\eps} | \Psi^\pm_{j,m} |^2 = 1$.
\end{definition}

The subsequent lemma is a classical result on Bloch expansions, see
e.g.\,\cite{PBL-1978}.

\begin{lemma}[Bloch expansion]\label{lem:Bloch-expansion}
  For $L, K \in\N$ and $\eps>0$ we consider the rectangle
  $W = (0,\eps L)\times (0,\eps K)$ and $u\in L^2(W; \C)$. For both
  eigenfunction families $(\Psi^+_{j, m})_{j,m}$ and
  $(\Psi^-_{j, m})_{j,m}$ the function $u$ possesses a unique
  expansion with coefficients $\alpha_{j,m}^\pm \in\C$ and convergence
  in $L^2(W;\C)$:
  \begin{equation}\label{eq:bloch-exp-1}
    u(x) = \sum_{j\in Q_{L}\times Q_{K}} \sum_{m=0}^\infty
    \alpha^\pm_{j, m}  \Psi^\pm_{j, m}(x)\, e^{2\pi \ri j\cdot x / \eps}\,.
  \end{equation}
\end{lemma}

We use the index-set $I_{L,K} := \{(j,m) | j\in Q_{L}\times Q_{K},\
m\in\N_0\}$ and multi-indices $\lambda = (j,m)\in I_{L,K}$, and define
for $x\in (0, \eps L)\times (0, \eps K)$
\begin{align}
  \label{eq:abbreviationUlambda}
  U^\pm_\lambda(x) := \Psi^\pm_\lambda (x)\, e^{2\pi \ri j\cdot x/\eps}\,.
\end{align}
With this notation, \eqref {eq:bloch-exp-1} simplifies to
\begin{equation}
  \label{eq:discrete-v-eps-Bloch}
  u(x) = \sum_{\lambda \in I_{L,K}} \alpha^\pm_\lambda U^\pm_\lambda (x)\,.
\end{equation}

We note that the Bloch basis functions inherit orthogonality
properties from the eigenfunctions $\Psi^\pm_{j,m}$.  Given
$L, K\in \N$, we calculate on $W = (0,\eps L)\times (0, \eps K)$ for
$\lambda = (j,m)$ and $\tilde\lambda = (\tilde j,\tilde m)$,
$\lambda, \tilde{\lambda} \in I_{L,K}$:
\begin{align*}
  \int_W \bar U^+_\lambda(x)  U^+_{\tilde\lambda}(x) \, dx
  = \int_W \Psi^+_\lambda (x) \Psi^+_{\tilde\lambda} (x) 
  e^{2\pi \ri (\tilde j - j) \cdot x/\eps}\, dx 
  = (\eps^2 L K)\ \delta_{\lambda, \tilde\lambda}\,.
\end{align*}
Indeed, for $j\neq \tilde j$, the expression vanishes by Lemma A.1 of
\cite {Lamacz-Schweizer-Bloch}. For $j= \tilde j$ and
$m\neq \tilde m$, the expression vanishes by orthogonality of the
different eigenfunctions $\Psi^+_\lambda$ to one $j$. In the remaining
case $\lambda = \tilde\lambda$, the statement is a consequence of the
normalization of $\Psi^+_\lambda$.  The same applies to
$(U^+_\lambda)_\lambda$.  Due to the $L^2(W)$-orthogonality, we have
the Plancherel formula
\begin{align}
  \label{eq:L2normexpansion}
  \|u\|^2_{L^2(W)} = \eps^2 LK \sum_{\lambda \in
    I_{L,K}}|\alpha^\pm_\lambda|^2\,.
\end{align}

\subsubsection*{Poynting numbers, index sets and projections}

We study a fixed index
$\lambda \in I := \{ \lambda = (j,m) | j\in [0,1]^2, m\in \N_0\}$. To
the index $\lambda$ we associate two Bloch waves, $U_\lambda^+$ and
$U_\lambda^-$ for the right domain and the left domain,
respectively. The Bloch waves $U_\lambda^\pm$ can transport energy to
the left or to the right; we now introduce the Poynting numbers
$P^\pm_\lambda$ which indicate the direction of energy transport.  The
sign of $P^+_\lambda$ coincides with the sign of the first component
of the group velocity, see Theorem 3 in \cite{FlissJoly2016} and the
explanation in Section 3.1. of \cite{Lamacz-Schweizer-Bloch}.  We set
\begin{equation}
  \label{eq:Poynting-def}
  P^\pm_\lambda := \Im\, \meanint_{Y_\eps} \bar U^\pm_\lambda (x)\,
  e_1\cdot \left[a_\pm(x)\nabla U^\pm_\lambda(x)\right]\, dx\,.
\end{equation}

\begin{definition}[Projections]\label{def:projection}
  Let $u\in L^2(W; \C)$ be a function on the rectangle $W = (0,\eps
  L)\times (0,\eps K)$ with the discrete Bloch expansion \eqref
  {eq:discrete-v-eps-Bloch}.  We define the projections $\Pi^\pm_{>0}$
  onto right-going Bloch waves by
  \begin{equation}\label{eq:projection}
    \left( \Pi^\pm_{>0} u \right) (x) 
    := \sum_{\stackrel{\lambda\in I_{L,K}}{P_\lambda^\pm>0}}
    \alpha^\pm_\lambda U^\pm_\lambda(x)\,.
  \end{equation}
  Projections $\Pi^\pm_{<0}$ onto left-going Bloch waves are defined
  accordingly.
\end{definition}

\paragraph{Sesquilinear forms $b^\pm$.}

We consider two functions $u\in L^2(W,\C)$ and $v\in H^1(W,\C)$ on the
rectangle $W = (0,\eps L)\times (0,\eps K)$. Two energy-flux
sesquilinear forms are defined by
\begin{equation}
  \label{eq:b-alt}
  b^\pm(u,v) :=\meanint_{W}\bar u(x)\,
    e_1\cdot\left[a_\pm(x)\nabla v(x)\right]\,dx\,.
\end{equation}
The connection to the Poynting number $P^\pm_\lambda$ of \eqref
{eq:Poynting-def} is expressed by
\begin{align}
  \label{eq:relationbS-2}
  P^\pm_\lambda 
  = \Im\, b^\pm\left(U^\pm_\lambda, U^\pm_\lambda\right)\,.
\end{align}
The following lemma has been shown in \cite {Lamacz-Schweizer-Bloch}
for $L = K$, the proof in the general case needs only notational
changes.

\begin{lemma}[Orthogonality property of $b^\pm$]
  \label{lem:orthogonalwavenumber}
  Given $L, K \in \N$, let $\lambda, \tilde\lambda\in I_{L,K}$ be two
  indices with $\lambda=(j,m)$, $\tilde\lambda=(\tilde j,\tilde m)$
  and $j\neq \tilde j$.  Then the basis functions $U^\pm_\lambda$ and
  $U^\pm_{\tilde\lambda}$ of \eqref {eq:abbreviationUlambda} satisfy
  \begin{align}
    \label{eq:orthogonalwavenumber}
    b^\pm(U_\lambda^\pm, U_{\tilde\lambda}^\pm) = 0 \,.
  \end{align}
\end{lemma}

\subsection{Problem (P)}

We can now formulate the truncated problem (P). We propose this
problem on a bounded domain as a replacement of the Helmholtz equation
with a radiation condition.

Our aim is to modify problem (P$_0$), which was defined by equations
\eqref {eq:Waveguide-R-L}--\eqref {eq:flux-R-L}. In problem (P) we
keep equation \eqref {eq:Waveguide-R-L}. Equation \eqref
{eq:outgoing-R-L} is strengthened, see \eqref {eq:outgoing-R-L-eta}
below.  The continuity condition \eqref {eq:weak-cont} is kept and the
flux condition \eqref {eq:flux-R-L} is weakened, see
\eqref{eq:flux-R-L-eta}.

The approximate problem is designed by choosing two index sets
$I^+ \subset I_{L,K}$ and $I^- \subset I_{L,K}$. We recall that every
element $\lambda \in I^\pm$ is of the form $\lambda = (j,m)$ with
$j_1\in Q_L \subset [0,1]$, $j_2\in Q_K \subset [0,1]$, $m\in \N$.
The index sets $I^\pm$ are chosen with the property
\begin{equation}\label{eq:Ipm-poynt}
  \lambda\in I^+ \Rightarrow P^+_\lambda > 0,
  \quad \lambda\in I^- \Rightarrow P^-_\lambda < 0,
\end{equation}
i.e. Bloch waves of the right radiation box travel to the right and
Bloch waves of the left radiation box travel to the left. Moreover, in
the numerics $I^\pm$ are finite. For given sets $I^\pm$ we define the
space
\begin{equation}
  \label{eq:X_eta}
  X^\pm :=  \spa \{ U^\pm_\lambda \,|\,  \lambda \in I^\pm \}\,.
\end{equation}

In order to approximate the Helmholtz equation on the unbounded domain
we first choose $R,L,\delta > 0$, and the two index sets
$I^\pm$. Given these parameters, the aim is to find $u: \Omega_{R+L}
\to \C$ that satisfies the following four conditions: (i) The
Helmholtz equation \eqref {eq:Waveguide-R-L} on $\Omega_R$. (ii) The
radiation condition \eqref {eq:outgoing-R-L} in the strengthened form
\begin{equation}
  \label{eq:outgoing-R-L-eta}
  \left\{ u\right\}^\pm_{R,L} \in X^\pm \,.
\end{equation}
We recall that, when $I^\pm$ consists only of indices of outgoing
waves, the projections \eqref {eq:outgoing-R-L} onto incoming waves
automatically vanish.  (iii) The weak continuity condition \eqref
{eq:weak-cont}.  (iv) A continuity condition that replaces \eqref
{eq:flux-R-L}; we obtain this condition as the natural interface
condition in a variational formulation of the problem.

Problem (P) will be made precise in Definition \ref {def:problem-P}
below. Essentially, the aim is to find $u$ that satisfies (i)--(iv).

\subsubsection*{The variational formulation}

In order to impose conditions (ii) and (iii), we seek $u$ in the
infinite dimensional function space
\begin{equation}
  \label{eq:V_eta}
  V := \left\{ u\in H^1(\Omega_{R+L}) \left|
      \phantom{\int}\!\!\!\!  
      u \text{ vertically periodic, }
      \left\{ u\right\}^+_{R,L} \in
      X^+, \left\{ u \right\}^-_{R,L} \in X^-
    \right. \right\}\,.
\end{equation}
Note that $V$ depends on the choice of the index sets $I^\pm$.

We now formulate (i) (the Helmholtz equation \eqref
{eq:Waveguide-R-L}) in a weak form, and, at the same time, encode the
flux condition (iv). In order to make integration by parts possible,
we introduce the special cut-off function $\vth : \R\to [0,1]$,
defined by
\begin{align*}
  \vth(\xi) :=
  \begin{cases}
    0 &\text{for } |\xi| \ge \eps(R+L),\\
    1 &\text{for } |\xi| \le \eps R,\\
    (\eps (R+L) - |\xi|)/(\eps L)\quad &\text{else,}
  \end{cases}
\end{align*}
compare Figure \ref {fig:cut-off}.  We regard $\vth$ also as a
function on two-dimensional domains such as $\Omega_{R+L}$ by setting
$\vth(x_1, x_2) = \vth(x_1)$.
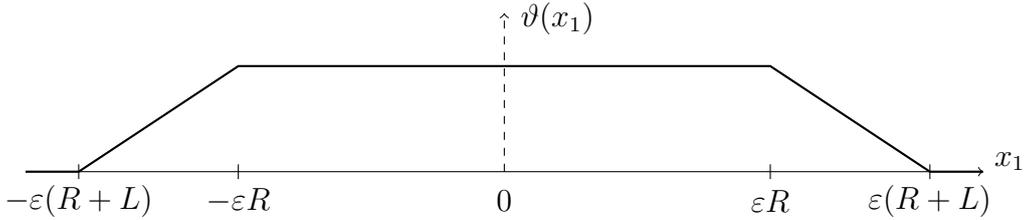
\begin{figure}[ht]
  \centering
  \begin{tikzpicture}[scale = 0.7]
    \draw[->] (-9,0)--(9,0);
    \draw[thick] (-9,0)--(-8,0)--(-5,2)--(5,2)--(8,0)--(9,0);

    \draw[-] (-8,-0.14)--(-8,0.14);
    \draw[-] (-5,-0.14)--(-5,0.14);
    \draw[-] (5,-0.14)--(5,0.14);
    \draw[-] (8,-0.14)--(8,0.14);

    \draw[->,dashed] (0,0)--(0,3);

    \node[] at (8,-.55) {$\eps (R+L)$};
    \node[] at (5,-.55) {$\eps R$};
    \node[] at (-8,-.55) {$-\eps (R+L)$};
    \node[] at (-5,-.55) {$-\eps R$};
    \node[] at (0, -.55) {$0$};

    \node[] at (1, 2.9) {$\vth(x_1)$};
    \node[] at (9.5, 0.2) {$x_1$};
  \end{tikzpicture}
  \caption{The cut-off function $\vth$}
  \label{fig:cut-off}
\end{figure}

In order to motivate the central definition of this article, we take
the complex conjugate of the Helmholtz equation \eqref
{eq:Waveguide-R-L} and multiply with the product $v\, \vth$, where
$v\in V$ is arbitrary. We obtain
\begin{align*}
  &\int_{\Omega_{R+L}} a \nabla \bar u\cdot \nabla (v\, \vth) 
  - \int_{\Omega_{R+L}} (1 - \ri\delta \one_{\Omega_R})\, \omega^2
  \bar u\, v\, \vth
  = \int_{\Omega_{R+L}} \bar f\, v\,\vth\,.
\end{align*}
The gradient of $\vth$ can be expressed explicitly as $\nabla\vth =
-(\eps L)^{-1} e_1$ on $W_{R,L}^+$, $\nabla\vth = (\eps L)^{-1} e_1$
on $W_{R,L}^-$, and $\nabla\vth = 0$ on $\Omega_R$. We use the above
relation to define an approximate problem.

\begin{definition}[Problem (P)]\label{def:problem-P} 
  Given $R,L,\delta>0$, the index sets $I^\pm$, and $f\in L^2(\Omega)$
  with support in $\Omega_R$, a function $u\in V$ is called a solution
  to problem (P) if
  \begin{equation}
    \label{eq:P}
    \begin{split}
      \beta(u,v) := 
      &\int_{\Omega_{R+L}} a \nabla \bar u\cdot \nabla v\, \vth -
      \int_{\Omega_{R+L}} (1 - \ri\delta \one_{\Omega_R})\, \omega^2
      \bar u\, v\, \vth\\
      &\quad - \frac1{\eps L} \int_{W_{R,L}^+} a \nabla \bar u\cdot e_1\, v
      + \frac1{\eps L} \int_{W_{R,L}^-} a \nabla \bar u\cdot e_1\, v =
      \int_{\Omega_R} \bar f\, v
    \end{split}
  \end{equation}
  holds for every $v\in V$.
\end{definition}

\begin{remark}\label{rem:equiv} Problem (P) is formally equivalent to the
  Helmholtz equation \eqref {eq:Waveguide-R-L}. More precisely, the
  following holds:

  Let $u$ be a solution (P). Then $u$ solves the Helmholtz
  equation \eqref{eq:Waveguide-R-L} on $\Omega_{R}$. 

  Let $u\in H^1(\Omega_{R+L})$ be a solution of the Helmholtz equation
  \eqref{eq:Waveguide-R-L} on $\Omega_{R+L}$ with $\delta$ replaced by
  $\delta \one_{\Omega_R}$ and with the source $f$ supported in
  $\Omega_R$.  Then $u$ satisfies \eqref {eq:P} (but not necessarily
  $u\in V$).
\end{remark}

\begin{proof}
  Regarding the first statement, we consider an arbitrary
  test-function $v \in C_c^\infty(\Omega_R)$. Then $v\, \vth = v$ and
  $\nabla v\, \vth = \nabla v$, integrals over $W_{R,L}^\pm$
  vanish. Therefore \eqref {eq:P} is nothing but the weak formulation
  of \eqref {eq:Waveguide-R-L}.

  To verify the second statement, it suffices to take the conjugate
  complex of \eqref {eq:Waveguide-R-L}, to multiply with $v\, \vth$
  and to integrate. No boundary terms appear in the integration by
  parts since $\vth$ vanishes for $x_1 = \pm (R+L)$.
\end{proof}

\paragraph{The coupling condition in a special case.}
Let us investigate solutions to (P) in the case $\delta=0$, assuming
that $X^\pm$ is spanned by Bloch waves $U_\lambda^\pm$ that have
exactly the eigenvalue $\omega^2$.  We can integrate by parts in
\eqref {eq:P} and obtain
\begin{align*}
  &- \int_{\Omega_{R+L}} \nabla\cdot (a \nabla \bar u) v\, \vth 
  - \int_{\Omega_{R+L}} \omega^2 \bar u\, v\, \vth\\
  &\quad\qquad - \int_{\Gamma_R^+} [e_1\cdot a\nabla \bar u]_{\Gamma_R^+}\, v
  + \int_{\Gamma_R^-} [e_1\cdot a\nabla \bar u]_{\Gamma_R^-}\, v
  = \int_{\Omega_R} \bar f\, v\,.
\end{align*}
The function $u$ solves the Helmholtz equation in $\Omega_R$ by Remark
\ref {rem:equiv}. On the other hand, as a linear combination of
solutions, $u$ solves the Helmholtz equation also in
$W_{R,L}^\pm$. This implies that the first two integrals cancel with
the right hand side. Since $v$ was arbitrary in $V$, we have
\begin{equation}
  \label{eq:flux-R-L-eta}
  \int_{\Gamma_R^\pm} [e_1\cdot a\nabla u]_{\Gamma_R^\pm}\, U_\lambda^\pm = 0\quad
  \text{ for every }\ U_\lambda^\pm\in X^\pm\,.
\end{equation}
In this sense, problem (P) implements a weak flux condition
that replaces \eqref {eq:flux-R-L}.

\section{Existence result}

Problem (P) of Definition \ref {def:problem-P} reads: Find $u\in V$
that satisfies
\begin{equation}
  \label{eq:beta-problem}
  \beta(u,v) = \int_{\Omega_R} \bar f\, v
  \qquad \forall v\in V\,.
\end{equation}
We will derive a coercivity result for the form $\beta$ and obtain, as
a corollary, an existence result for problem (P).  The coercivity
result will be based on the following assumptions.

\begin{assumption}\label{ass:assumptions}
  We introduce the following assumptions on the index sets $I^\pm$ and
  the corresponding spaces $X^\pm$.
  \begin{enumerate}
  \item[(A1)] Positive speed: There exists a positive number $c_0 > 0$
    such that, for every $\lambda$ with $U_\lambda^\pm \in X^\pm$,
    there holds
    \begin{equation}
      \label{eq:pos-speed}
      \pm P_\lambda^\pm \ge c_0\,.
    \end{equation}
  \item[(A2)] For every pair of indices $\lambda = (j,m),
    \tilde\lambda = (\tilde j,\tilde m) \in I^\pm$ the wave numbers
    are different: 
    $j\neq \tilde j$.
  \item[(A3)] Regularity: For some constant $C_0>0$ and every
    $u\in X^\pm$ holds
    \begin{equation}
      \label{eq:inverse}
      \| u \|_{H^1(W)}^2 \le C_0 \| u \|_{L^2(W)}^2 \,.
    \end{equation}
  \end{enumerate}
\end{assumption}

\begin{remark}
  On Assumption \ref {ass:assumptions}. (i) Assumption (A2) is only
  used to exploit $b^\pm(U_\lambda^\pm, U_{\tilde\lambda}^\pm) = 0$
  for $\tilde\lambda \neq \lambda$. The assumption is not essential
  for the numerical scheme. (ii) If Assumption (A2) is satisfied, then
  the sets $I^+$ and $I^-$ are necessarily finite and the spaces
  $X^\pm$ are finite dimensional.  (iii) When $I^+$ and $I^-$ are
  finite sets, then (A3) is automatically satisfied since all
  functions $U_\lambda^\pm$ possess $H^1$-regularity. Under the same
  assumption, (A1) is satisfied if and only if no wave $U_\lambda^\pm$
  is used in $X^\pm$, which travels in vertical direction.
\end{remark}

\begin{theorem}[Existence result for problem (P)]
  \label{thm:existence}
  Let $R, L, \delta$ be positive parameters and let $f\in L^2(\Omega)$
  be a function with support in $\Omega_R$.  Let the coefficient
  function $a\in L^\infty(\Omega; \R)$ be bounded from below by
  $a_0>0$ and identical to $Y_\eps$-periodic functions $a_\pm$ for
  $\pm x_1 > \eps R/2$.  Let the index sets $I^\pm$ satisfy properties
  (A1)--(A3) of Assumption \ref {ass:assumptions} with constants
  $c_0, C_0 > 0$. Then problem (P) of Definition \ref {def:problem-P}
  has a unique solution $u$. For a constant
  $C = C(R, L, a_0, \delta, c_0, C_0)$ we have the stability estimate
  \begin{equation}
    \label{eq:sol-est}
    \| u \|_{H^1(\Omega_{R+L})} \le C \| f \|_{L^2(\Omega_{R})}\,.
  \end{equation}
\end{theorem}

We derive the above theorem with a constant $C$ that satisfies $C\sim
\delta^{-1}$ for small $\delta$. The numerical experiments show a much
better behavior of the solution $u$: The scheme has good convergence
properties even for $\delta = 0$.

\begin{proof}[Proof of Theorem \ref {thm:existence}]
  The aim on the next pages is to derive, for two numbers $\sigma,
  \gamma > 0$, a coercivity estimate of the form
  \begin{equation}
    \label{eq:coercivity}
    \Im\, \beta(u,u) 
    + \sigma\delta\  \Re\, \beta(u,u) 
    \ge \gamma \| u\|_{H^1(\Omega_{R+L})}^2\,.
  \end{equation}
  We obtain this result as relation \eqref {eq:coercivity-H1-RL} in
  Proposition \ref {prop:H1RL-lower}.

  The Lax-Milgram Lemma implies the existence statement of Theorem
  \ref {thm:existence}. We note that the Lax-Milgram lemma in complex
  Hilbert spaces is applicable for sesquilinear forms that satisfy a
  coercivity estimate of the form \eqref {eq:coercivity}. We refer to
  \cite {Alt-FA} for a proof of the Lax-Milgram Lemma that works with
  the coercivity assumption $| \beta(u,u) | \ge \gamma \| u \|_{H^1}^2$,
  which is implied by \eqref {eq:coercivity}.

  Let us recall here the main point of the proof, which is the
  derivation of estimate \eqref {eq:sol-est} for solutions of \eqref
  {eq:beta-problem}: Using $v = u\in V$ as a test-vector in \eqref
  {eq:beta-problem} and exploiting \eqref {eq:coercivity} yields
  \begin{align*}
    \gamma \| u\|_{H^1(\Omega_{R+L})}^2
    &\le \Im\, \beta(u,u) 
    + \sigma\delta\  \Re\, \beta(u,u) 
    \le (1 + \sigma\delta) | \beta(u,u) |\\
    &\le (1 + \sigma\delta) \left| \int_{\Omega_R} \bar f\, u \right|
    \le (1 + \sigma\delta) \| f\|_{L^2(\Omega_R)} \| u\|_{L^2(\Omega_R)} \,.
  \end{align*}
  This provides \eqref {eq:sol-est} with the constant $C = (1 +
  \sigma\delta) \gamma^{-1}$.
\end{proof}

\subsection{Coercivity in $L^2$}

The main feature of the bilinear form $\beta$ is the positivity of the
imaginary part. Moreover, the imaginary part controls certain norms of
the argument.

\begin{lemma}[$L^2$-coercivity]\label{lem:L2R-lower}
  Let the index sets $I^\pm$ satisfy the outgoing wave property
  \eqref{eq:Ipm-poynt}. Then the sesquilinear form $\beta$ of
  Definition \ref {def:problem-P} satisfies the following
  $L^2$-coercivity estimate:
  \begin{equation}
    \label{eq:coercivity-L2-R}
    \Im\, \beta(u,u) \ge \delta\omega^2 \| u\|_{L^2(\Omega_{R})}^2
    \qquad \forall u\in V\,.
  \end{equation}

  Let additionally the positive speed property \eqref {eq:pos-speed}
  be satisfied in $X^\pm$ with the constant $c_0>0$. Then we have
  the stronger estimate
  \begin{equation}
    \label{eq:coercivity-L2-W}
    \Im\, \beta(u,u) \ge 
    \delta\omega^2 \| u\|_{L^2(\Omega_{R})}^2
    + \frac{c_0}{\eps L}\, 
    \left( \| u\|_{L^2(W^+_{R,L})}^2 + \| u\|_{L^2(W^-_{R,L})}^2\right)\,.
  \end{equation}
\end{lemma}

\begin{proof} Let $u\in V$ be arbitrary.  By definition of the space
  $V$, the function $u$ can be expanded in Bloch waves in the two
  rectangles $W_{R,L}^\pm$.  We write the shifted functions as
  \begin{equation}
    \label{eq:expand-box}
    \left\{ u\right\}^+_{R,L} 
    = \sum_{\lambda\in I^+} \alpha_\lambda^+ U_\lambda^+\,,\qquad
    \left\{ u\right\}^-_{R,L}
    = \sum_{\lambda\in I^-} \alpha_\lambda^- U_\lambda^+\,.
  \end{equation}

  With $\beta$ from \eqref {eq:P} we now calculate $\beta(u,u)$. The
  integrals over $W_{R,L}^\pm$ can be expressed with the sesquilinear
  forms $b^\pm$ from \eqref{eq:b-alt}. The orthogonality property of
  Lemma \ref {lem:orthogonalwavenumber} allows to calculate
  \begin{align*}
    \beta(u,u) &= \int_{\Omega_{R+L}} a |\nabla u|^2\, \vth -
    \int_{\Omega_{R+L}} (1- \ri\delta \one_{\Omega_R})\, \omega^2 | u |^2 \vth\\
    &\quad - \eps K\, \overline{b^+}\left(\left\{ u\right\}^+_{R,L} ,
      \left\{ u\right\}^+_{R,L}\right)
    + \eps K\, \overline{b^-}\left(\left\{ u\right\}^-_{R,L} , \left\{ u\right\}^-_{R,L}\right)\\
    &= \int_{\Omega_{R+L}} a |\nabla u|^2\, \vth -
    \int_{\Omega_{R+L}} (1- \ri\delta \one_{\Omega_R})\, \omega^2 | u |^2 \vth\\
    &\quad - \eps K\, \sum_{\lambda\in I^+}
    |\alpha_\lambda^+|^2 \ \overline{b^+}(U_\lambda^+, U_\lambda^+) +
    \eps K\, \sum_{\lambda\in I^-} |\alpha_\lambda^-|^2 \
    \overline{b^-}(U_\lambda^-, U_\lambda^-)\,.
  \end{align*}
  Taking the imaginary part and using the definition of $P^\pm$ from
  \eqref {eq:relationbS-2}, we obtain
  \begin{equation}\label {eq:Im-beta-u-u}
    \Im\, \beta(u,u) = \int_{\Omega_R} \delta\, \omega^2 | u |^2
    + \eps K \sum_{\lambda\in I^{+}} |\alpha_\lambda^+|^2 P_\lambda^+
    - \eps K \sum_{\lambda\in I^{-}} |\alpha_\lambda^-|^2 P_\lambda^-\,.
  \end{equation}
  Non-negativity of $P_\lambda^+$ for $\lambda\in I^{+}$ and
  non-positivity of $P_\lambda^-$ for $\lambda\in I^{-}$ implies
  the lower bound \eqref {eq:coercivity-L2-R}.

  \smallskip {\em Estimate \eqref {eq:coercivity-L2-W}.}  In the case
  that the positive speed property is satisfied, the box integrals
  yield a strictly positive contribution.  Inserting \eqref
  {eq:pos-speed} into \eqref {eq:Im-beta-u-u} we find
  \begin{equation}
    \begin{split}
      \Im\, \beta(u,u) &\ge \int_{\Omega_R} \delta\, \omega^2 | u |^2
      +c_0 \eps K \left( \sum_{\lambda\in
          I^{+}} |\alpha_\lambda^+|^2
        + \sum_{\lambda\in I^{-}} |\alpha_\lambda^-|^2 \right)\\
      &\ge \int_{\Omega_R} \delta\, \omega^2 | u |^2
      +\frac{c_0}{\eps L} \left( \int_{W_{R,L}^+} |u|^2 +
        \int_{W_{R,L}^-} |u|^2 \right)\,,
    \end{split}\label{eq:L2-intermediate-834}
  \end{equation}
  where we used the Plancherel formula \eqref {eq:L2normexpansion} in
  the last line. This yields \eqref {eq:coercivity-L2-W}.
\end{proof}

\begin{remark}
  (i) The $L^2(\Omega_R)$ coercivity of Lemma \ref {lem:L2R-lower} is
  not sufficient for an existence result since the sesquilinear form
  $\beta$ is defined on $H^1(\Omega_{R+L})$.

  (ii) The lower bound in \eqref {eq:coercivity-L2-R} depends on
  $\delta$. This fact is discouraging when one seeks to perform a
  limiting absorption principle, which needs estimates that are
  uniform in $\delta$. By contrast, considering only the norm of the
  solution in the radiating boxes $W_{R,L}^\pm$, the bound in \eqref
  {eq:coercivity-L2-W} is independent of $\delta$. This
  $\delta$-independent bound seems to be the reason for the
  numerically observed stability of problem (P): The scheme works well
  even for $\delta = 0$.
\end{remark}

\subsection{Coercivity in $H^1$}

We turn now to the coercivity estimate that corresponds to the chosen
function space. The two assumptions in \eqref {eq:ass-b-sigma} and
\eqref {eq:ass-b-Re-Im} essentially demand the smallness of $\sigma$
in comparison to $1$ and to $c_0/\delta$.

\begin{proposition}[$H^1(\Omega_{R+L})$-coercivity]
  \label{prop:H1RL-lower}
  Let the index sets $I^{\pm}$ satisfy properties (A1)--(A3) of
  Assumption \ref {ass:assumptions} with constants $c_0, C_0 > 0$.
  Let $\sigma>0$ be small enough to satisfy the two properties
  \begin{align}
    \label{eq:ass-b-sigma}
    \sigma &\le  \min \left\{ \frac12 , 
      \frac{c_0}{4 \eps L \delta \omega^2} \right\}\,,\\
    2 \sigma\delta\ \left| \Re\, b^\pm(U_\lambda^\pm, U_\lambda^\pm) \right|
    &\le \Im\, b^\pm(U_\lambda^\pm, U_\lambda^\pm)
    \qquad\forall \lambda \text{ with } U_\lambda^\pm \in X^\pm\,.
    \label{eq:ass-b-Re-Im}
  \end{align}
  Then there exists
  $\gamma = \gamma(c_0, \delta, \sigma, \omega, a_0) >0$ such that for
  every $u\in V$ holds
  \begin{equation}
    \label{eq:coercivity-H1-RL}
    \Im\, \beta(u,u) 
    + \sigma\delta\  \Re\, \beta(u,u) 
    \ge \gamma \| u\|_{H^1(\Omega_{R+L})}^2\,.
  \end{equation}
\end{proposition}

\begin{proof}
  Relation \eqref {eq:coercivity-L2-W} of Lemma \ref {lem:L2R-lower}
  together with \eqref {eq:Im-beta-u-u} provides the following lower
  bound for the imaginary part of $\beta(u,u)$:
  \begin{equation}
    \label{eq:L2-result-2354}
    \begin{split}
      \Im\, \beta(u,u) \ge\, & \delta\omega^2
      \| u\|_{L^2(\Omega_{R})}^2 + \frac{c_0}{2\eps L}\, \left( \|
        u\|_{L^2(W^+_{R,L})}^2 + \| u\|_{L^2(W^-_{R,L})}^2\right)\\
      &+ \frac{\eps K}{2} \sum_{\lambda\in I^{+}} |\alpha_\lambda^+|^2 P_\lambda^+
      - \frac{\eps K}{2} \sum_{\lambda\in I^{-}} |\alpha_\lambda^-|^2 P_\lambda^-\,.
    \end{split}
  \end{equation}
  We now evaluate the real part of $\beta(u,u)$ from its definition in
  \eqref {eq:P}. After a multiplication with the factor $\sigma\delta$
  we find, with the orthogonality \eqref {eq:orthogonalwavenumber},
  \begin{equation}
    \label{eq:L2-result-real}
    \begin{split}
      &\sigma\delta\ \Re\, \beta(u,u) = \,
      \sigma\delta\, \int_{\Omega_{R+L}} a |\nabla  u|^2\, \vth 
      - \sigma\delta\, \int_{\Omega_{R+L}} \omega^2 |u|^2 \vth\\
      &\quad - \sigma\delta\, \eps K\, 
      \sum_{\lambda\in I^{+}} |\alpha_\lambda^+|^2 \,
      \Re\, b^+(U_\lambda^+, U_\lambda^+) + \sigma\delta\, \eps K\,
      \sum_{\lambda\in I^{-}} |\alpha_\lambda^-|^2 \,
      \Re\, b^-(U_\lambda^-, U_\lambda^-)\,.
    \end{split}
  \end{equation}
  Due to assumption \eqref {eq:ass-b-Re-Im} on $\sigma$, the two sums
  in \eqref {eq:L2-result-real} are smaller in absolute value than the
  two sums in \eqref {eq:L2-result-2354}.  Due to assumption \eqref
  {eq:ass-b-sigma} on $\sigma$, the second integral in \eqref
  {eq:L2-result-real} is bounded in absolute value by the half of the
  first two contributions on the right hand side of \eqref
  {eq:L2-result-2354}. We therefore obtain 
  \begin{equation*}
    \begin{split}
      \Im\, \beta(u,u) + \sigma\delta\ \Re\, \beta(u,u) 
      \ge
      &\frac{\delta\omega^2}{2} \| u\|_{L^2(\Omega_{R})}^2 +
      \frac{c_0}{4\eps L}\, \left( \| u\|_{L^2(W^+_{R,L})}^2 + \|
        u\|_{L^2(W^-_{R,L})}^2\right)\\
      & + \sigma\delta a_0 \int_{\Omega_{R+L}} |\nabla  u|^2\, \vth \,,
    \end{split}
  \end{equation*}
  where $a \ge a_0$ was used.  The inverse estimate \eqref
  {eq:inverse} implies that the right hand side controls the squared
  $H^1(\Omega_{R+L})$-norm. We thus arrive at \eqref
  {eq:coercivity-H1-RL}.
\end{proof}

\begin{remark}
  An inspection of the assumptions on $\sigma$ shows that the
  coercivity constant $\gamma$ has the properties $\gamma\sim \delta$
  for small $\delta>0$ and $\gamma\sim \sigma \sim c_0$ for small
  $c_0 >0$.
\end{remark}


\section{Numerical method and examples}
\label{sec.numerics}

\subsection{Numerical method}

\subsubsection*{Finite element discretization of problem (P)}

Our aim is to approximate problem (P) with a finite element method
(FEM), using an enriched, problem adapted set of basis functions. Once
a finite dimensional subspace $V_h$ of $V$ is defined with basis
functions, we have obtained in a natural way a discretization of
problem (P). Our construction uses piecewise linear hat functions on a
triangular mesh. More precisely, the space $V_h$ is spanned by
standard (piecewise linear) hat functions in $\Omega_R$ and by
(approximation of) Bloch waves in the radiation boxes. We use Bloch
waves $U_\lambda^+$ with a positive Poynting number in the radiation
box $W_{R,L}^+$ and $U_\lambda^-$ with a negative Poynting number in
the box $W_{R,L}^-$. The Bloch waves themselves are computed with
piecewise linear hat functions.

\paragraph{Choice of a regular grid.}

We use a (uniform) triangulation mesh with right angled triangles on
$\overline{\Omega_{R+L}}$. The fineness parameters $h_1>0$ and $h_2>0$
denote the lengths of the triangle legs in the $x_1$ and $x_2$
directions respectively. The grid points $x^{(k)}\in
\overline{\Omega_{R+L}}$, $k=1,\dots, N_h$ are enumerated so that
$$
x^{(k)} \in \begin{cases} (-\eps R,\eps R)\times [0,\eps K) &\text{ for } k=1,\dots,N_0, \\
\overline{W^+_{R,L}} &\text{ for } k=N_0+1,\dots,N_0+N_W,\\
\overline{W^-_{R,L}} &\text{ for } k=N_0+N_W+1,\dots,N_0+2N_W=N_h.\\ 
\end{cases}
$$

\paragraph{Hat functions.}

To the grid we assign the standard piecewise linear hat functions
$\phi_k,k=1,\dots,N_h$ with $\phi_k(x^{(l)})=\delta_{k,l}$ for
$k,l=1,\dots,N_h$. To impose vertical periodicity in $\Omega_R$, each
hat function $\phi_k$ with
$x^{(k)}\in [-\eps (R+L),\eps (R+L)]\times \{0\}$ (i.e.\,a lower
boundary point) consists of the hat function half corresponding to
$x^{(k)}$ and the hat function half corresponding to the artificial
grid point $(x_1^{(k)},x_2^{(k)}+\eps K)$.  These hat functions (and,
hence, any linear combination thereof) are periodic in the vertical
direction.


\paragraph{Bloch waves.}

Next, we define the Bloch wave basis functions. For each selected wave
vector $j$ we first solve the eigenvalue problems
$\mathcal{L}^\pm_j \Psi_j^\pm=\mu^\pm(j)\Psi_j^\pm$ on the cube
$Y_\eps$ with periodic boundary conditions. The FEM-solutions to
these problems are denoted $\Psi_{j,m}^{\pm,h}$, $m\in \N$.

For each $m$ in an appropriately chosen subset of $\N$ we extend the
solution by periodicity onto the radiation box
$\overline{W^\pm_{R,L}}$ and use \eqref{eq:abbreviationUlambda} to
define $U_{\lambda}^{\pm,h}$ (recall that $\lambda=(j,m)$).  Each
Bloch wave $U_{\lambda}^{\pm,h}$ is extended by zero to all grid
points outside $\overline{W^\pm_{R,L}}$. Our selected set of indices
$\lambda$ is denoted by $I^{\pm,h}$ and is specified below.

The resulting Bloch waves can be written, for each
$\lambda_i\in I^{\pm,h}$, as
\begin{equation}\label{E:Bloch-hat}
  U_{\lambda_i}^{+,h}(x)=\sum_{k=N_0+1}^{N_0+N_W}\kappa_k^{+,i}\phi_k(x),
  \quad U_{\lambda_i}^{-,h}(x)=\sum_{k=N_0+N_W+1}^{N_h}\kappa_k^{-,i}\phi_k(x)\,,
\end{equation}
with coefficients $\kappa_k^{\pm,i}\in \C$ for all $i,k$. We emphasize
that the functions $U_{\lambda_i}^{\pm,h}$ are continuous on
$\overline{\Omega_{R+L}}$ for all $i$. On the other hand, they are
concentrated in the radiation boxes in the sense that
$U_{\lambda_i}^{+,h}(x^{(k)})=0$ for $x^{(k)}_1<\eps R$,
$U_{\lambda_i}^{-,h}(x^{(k)})=0$ for $x^{(k)}_1>-\eps R$.

The sets $I^{\pm,h}$ are discrete analogs of $I^\pm$ satisfying
(A1)--(A3) in Assumption \ref{ass:assumptions}. However, in contrast
to $I^\pm$ we choose for the numerics the $j-$domain to be
$\mathbb{B}:=\left(-\tfrac{1}{2},\tfrac{1}{2}\right]^2,$ such that
$\frac{2\pi}{\eps}\mathbb{B}$ is the standard Brillouin zone
corresponding to the periodicity cell $Y_\eps$. This symmetric choice
has the advantage that the band structure plots clearly show the
conical shape in the case of homogeneous media. The set $Q_{K}$ from
Section \ref{S:Bloch} (defined to ensure the vertical
$\eps K$-periodicity of the solution in $W^\pm_{R,L}$) needs to be
modified to
$$ 
Q'_K:=\begin{cases}
\{-\frac{K-2}{2K},-\frac{K-4}{2K},\dots, \frac{1}{2}\} , & \text{for } K\in 2\N,\\
\{-\frac{K-1}{2K},-\frac{K-3}{2K},\dots, \frac{K-1}{2K}\} , & \text{for } K\in 2\N+1.
\end{cases}
$$ 
The Poynting numbers $P_\lambda^\pm$ are computed via a numerical
quadrature of \eqref{eq:Poynting-def} in the piecewise linear finite
element space.

\paragraph{Approximation of the space $V$.}

Assuming for simplicity that the number of Bloch basis functions is
the same in both radiation boxes,
$|I^{+,h}|= |I^{-,h}|=: N_{\text{Bl}}\in \N$, we can now define the
finite dimensional space as
\begin{equation}\label{E:Vh}
V_h:= \text{span}\{\psi_1, \dots,\psi_{N_0+2N_{\text{Bl}}}\},
\end{equation}
where 
$$\psi_k:=\begin{cases}
\phi_k \ &\text{for} \  k=1, \dots N_0,\\
U_{\lambda_i}^{+,h} \ \text{with} \ \lambda_i\in I^{+,h}  \ &\text{for} \ k=N_0+i, i= 1,\dots, N_{\text{Bl}},\\
U_{\lambda_i}^{-,h} \ \text{with} \ \lambda_i\in I^{-,h}  \ &\text{for} \ k=N_0+N_{\text{Bl}}+i, i= 1,\dots, N_{\text{Bl}}.
\end{cases}$$

\paragraph{Discretization of problem (P).}

The finite dimensional subspace $V_h \subset V$ defines a discrete
problem (P). The complex conjugate of \eqref{eq:P} can be written with
matrices and coordinate vectors as
$${\bf A}\vec{U}-{\bf B}\vec{U}-\omega^2{\bf M}^{(\delta)}\vec{U}=\vec{F}\,.$$
Here $\vec{U} \in \C^{N_0+2N_{\text{Bl}}}$ is the unknown coordinate
vector and the matrix entries are, for each
$k,l=1,\dots,N_0+2N_{\text{Bl}}$,
$$
\begin{array}{rlrl}
{\bf A}_{k,l}&=\int_{\Omega_{R+L}} a\vartheta \nabla \overline{\psi_k}\cdot \nabla \psi_l, &{\bf B}_{k,l} &=\frac{1}{\eps L}\left(\int_{W^+_{R,L}}a\overline{\psi_k}e_1 \cdot \nabla \psi_l-\int_{W^-_{R,L}}a\overline{\psi_k}e_1 \cdot \nabla \psi_l\right),\\
{\bf M}^{(\delta)}_{k,l}&=(1+\ri \delta{\bf 1}_{\Omega_R})\int_{\Omega_{R+L}}\overline{\psi_k}\psi_l, &F_k &= \int_{\Omega_R} f\overline{\psi_k}.
\end{array}
$$
Due to the representation \eqref{E:Bloch-hat}, all integrals involving
the Bloch waves $U^{\pm,h}_{\lambda_i}$ can be evaluated using solely
integrals of hat functions $\phi_k$.

\subsubsection*{Numerical implementation caveats}

\paragraph{Choice of the Bloch indices $I^{\pm,h}$.}

A suitable choice of the index sets $I^{\pm,h}$ is crucial for an
efficient and accurate numerical scheme. A direct analog of $I^\pm$
satisfying $I^\pm\subset I_{L,K}$ and assumptions (A1)--(A2) can be
built by the following procedure. First, for each
$j\in Q'_L\times Q'_K$ one solves the Bloch eigenvalue problems
$\mathcal{L}^\pm_j \Psi_j^\pm=\mu^\pm(j)\Psi_j^\pm$ in the
FEM-approximation. For each $j$ one keeps only the eigenvalue
$\mu_m^\pm(j)$ closest to $\omega^2$, which selects a natural number
$m$ for every vector $j$. Subsequently, one filters out eigenvalues
$\mu^+_m(j)$ with a non-positive Poynting number of the Bloch wave and
eigenvalues $\mu^-_m(j)$ with a non-negative Poynting number of the
Bloch wave. The remaining pairs $(j,m)$ define the sets $I^{\pm,h}$.

Numerical tests have shown that this approach works, but more accurate
results are obtained when the horizontal $L$-periodicity requirement
(i.e.\,$j_1\in Q_L'$) is dropped. We take the liberty to choose points
$(j,m)\in I^{\pm,h}$ so that the frequency level $\omega^2$ is better
realized: $|\omega^2-\mu^\pm_m(j)|$ is minimized.

In practice, we first solve the eigenvalue problems
$\mathcal{L}^\pm_j \Psi_j^\pm=\mu^\pm(j)\Psi_j^\pm$ in the
FEM-approximation for all $j$ (with $j_2\in Q'_K$) on a selected
$j_1-$mesh of $(-1/2, 1/2]$. For each $j$ we save only the eigenvalue
$\mu_m^\pm(j)$ closest to $\omega^2$. Subsequently, for each
$j_2 \in Q'_K$ we search for intersections of the line
$\left(-\tfrac{1}{2},\tfrac{1}{2}\right]\times\{j_2\}$ with the level
set of the band structure at $\omega^2$. Such intersections can occur
for more eigenvalue families $\mu^\pm_m,m\in \N$. For each such family
$\mu^\pm_m$ the intersections are found by an interpolation producing
an approximation of the $j_1$-coordinates at which
$\mu^\pm_m(j_1,j_2)=\omega^2$. The resulting pair $(m,(j_1,j_2))$ is
then included in the set $I^{\pm,h}$ if the Poynting number of the
Bloch wave $U_{(j,m)}^{\pm,h}$ has the appropriate sign. Including
these intersection points in $I^{\pm,h}$ leads to the same accuracy of
the calculations with a much smaller number $N_\text{Bl}$.

\paragraph{Orthogonalization of the Bloch waves
  $U_{\lambda_i}^{\pm,h}$ with $\lambda_i \in I^{\pm,h}$.}

The index set $I^{\pm,h}$ selected by the above procedure typically
contains $j$-points lying no further than $\tfrac{\pi}{\eps K}$
apart. This separation is small for $K$ large and the corresponding
Bloch waves $U^{+,h}_{\lambda_i},\lambda_i\in I^{+,h}$ are similar if
they belong to the same $m$. In order to keep the condition number of
${\bf A}$ small, we $L^2$-orthonormalize the set
$U^{+,h}_{\lambda_i},\lambda_i\in I^{+,h}$ via the modified
Gram-Schmidt procedure.  The superscript ``$-$" is treated analogously.

\paragraph{Scattering problem with an incoming wave.}
 
When studying scattering problems with an incoming field
$u^{(\text{in})}$ supported on $x_1<0$, we use the following method to
transform $u^{(\text{in})}$ into the inhomogeneity $f$ (without
changing $f$ on $x_1\geq 0$). We set
$$u_\theta:= u-u^{(\text{in})}\theta \quad \text{with} \ \theta(x_1,x_2)=\theta(x_1)=\begin{cases}1, & x_1<-\eps R\\
  1-\tanh(d(x_1+\tfrac{\eps R}{2})), & x_1\in [-\eps R,0)\\ 0, &
  x_1\geq 0,\end{cases}$$ where $d$ is sufficiently large to ensure
that $\theta$ is close to zero at $x_1=0$. This leads to the
transformed problem
\begin{equation}\label{E:Helmh_incom}
  -\nabla \cdot (a\nabla u_\theta)-\omega^2(1+\ri \delta)u_\theta=\tilde{f},
  \quad \tilde{f}:=f+2a\nabla u^{(\text{in})}\cdot \nabla \theta+u^{(\text{in})}a\Delta \theta\,,
\end{equation}
which we treat exactly as described above.

\subsection{Numerical results I: Comparison with
  homogenization}\label{S:comp_hom}

In our first numerical example we consider the interface between a
homogenous and a periodic material and study a single incoming plane
wave $u^{(\text{in})} = e^{\ri j^{(\text{in})}\cdot x}$,
$j^{(\text{in})}\in \R^2$.  We use the method developed above to
calculate an approximate solution $u=u_\theta+u^{(\text{in})}$, where
$u_\theta$ solves \eqref{E:Helmh_incom}. We compare this solution $u$
of the original problem with the solution $u_\text{hom}$ of the
interface problem with a homogenized material on the right.

We choose $\eps=1$, $R=15$, $L=6$ and $H=14$ and a discretization
given by $h_1=0.05$ and $h_2\approx 0.0526$. We use $N_\text{Bl}=4$
Bloch waves in each radiation box. The absorption constant is set to
$\delta =10^{-4}$. The heterogeneous material is described by the
constant $1$ (``air") on $x_1<0$ and a periodic array of discs on
$x_1\geq 0$. We choose the same structure as in
\cite{PhysRevB.65.201104}, i.e.
\begin{equation}\label{E:a-Luo}
  a(x)=\begin{cases}
    1, \ &x_1<0\\
    a_+(x), \ &x_1\geq 0
  \end{cases}
\end{equation}
with
\begin{equation}
  a_+(x) := \begin{cases}
    \frac{1}{12}, &\text{dist}(x,\{(\tfrac{1}{2},0),(0,\tfrac{1}{2}),(\tfrac{1}{2},1),(1,\tfrac{1}{2})\})<\tfrac{1}{\sqrt{2}}0.35\\
    1 &\text{otherwise}
  \end{cases}
  \label{eq:a+formula}
\end{equation}
for $x\in Y_\eps$ and $a_+(x)=a_+(x+\eps e_j)$, $j=1,2$ for all $x\in
\R^2$ .

The incoming wave $u^{(\text{in})}(x)=e^{\ri j^{(\text{in})}\cdot x}$
has to satisfy $\omega^2=|j^{(\text{in})}|^2$ since $a\equiv 1$ on
$x_1<0$.

For $\eps\to 0$ (or, equivalently, for incoming waves of large period)
the problem on $x_1\geq 0$ can be approximated by the homogenized
equation \cite{Conca-Vanninathan,DLS-2014}
$$
\begin{aligned}
  &-a_*\Delta u_\text{hom}=\omega^2 u_\text{hom}\,,\\
  &a_*=\frac{1}{2}\left(\frac{\eps}{2\pi}\right)^2\partial_{j_1}^2\mu^+_0(0)
  =\mean_{Y_\eps}a_+(x)\left[1-\frac{\ri}{2}\left(\partial_{x_1}\partial_{k_2}\psi_0^+(x)+\partial_{x_2}\partial_{k_1}\psi_0^+(x)\right)
  \right]\, dx\,.
\end{aligned}
$$
Note that the homogenization coefficient $a_*$ is a scalar due to 
spatial symmetries of $a_+$.  The resulting numerical value for the
above discretization is
\begin{equation}\label{E:astar}
a_* \approx 0.1699\,.
\end{equation}

In our first example we choose the frequency $\omega$ and the incoming
wave $u^{(\text{in})}(x)=e^{\ri j^{(\text{in})}\cdot x}$ with
\begin{equation}\label{E:incom-sm-om}
  \omega =0.2\pi\approx 0.628, \quad j^{(\text{in})}\approx (0.440, 0.449)\,.
\end{equation}
Since the frequency is quite small (and, hence, the wavelength is
large), we can expect that the homogenized setting provides a good
approximation.

The band structure and the level set at $\omega$ is plotted in Figure
\ref{F:comp_hom_bd_str}, where also the points $j$ with
$(j,m)\in I^{+,h}$ are marked by black dots (all very close to each
other). At the selected frequency $\omega=0.2 \pi$ the band structure
is only a small perturbation of a cone, which means that the band
structure is similar to that of a homogenous medium. Indeed,
homogenization theory provides a good qualitative prediction, as shown
in Figure \ref{F:comp_hom_om_sm}.
\begin{figure}[ht!]
  \begin{center}
    \includegraphics[scale=0.6]{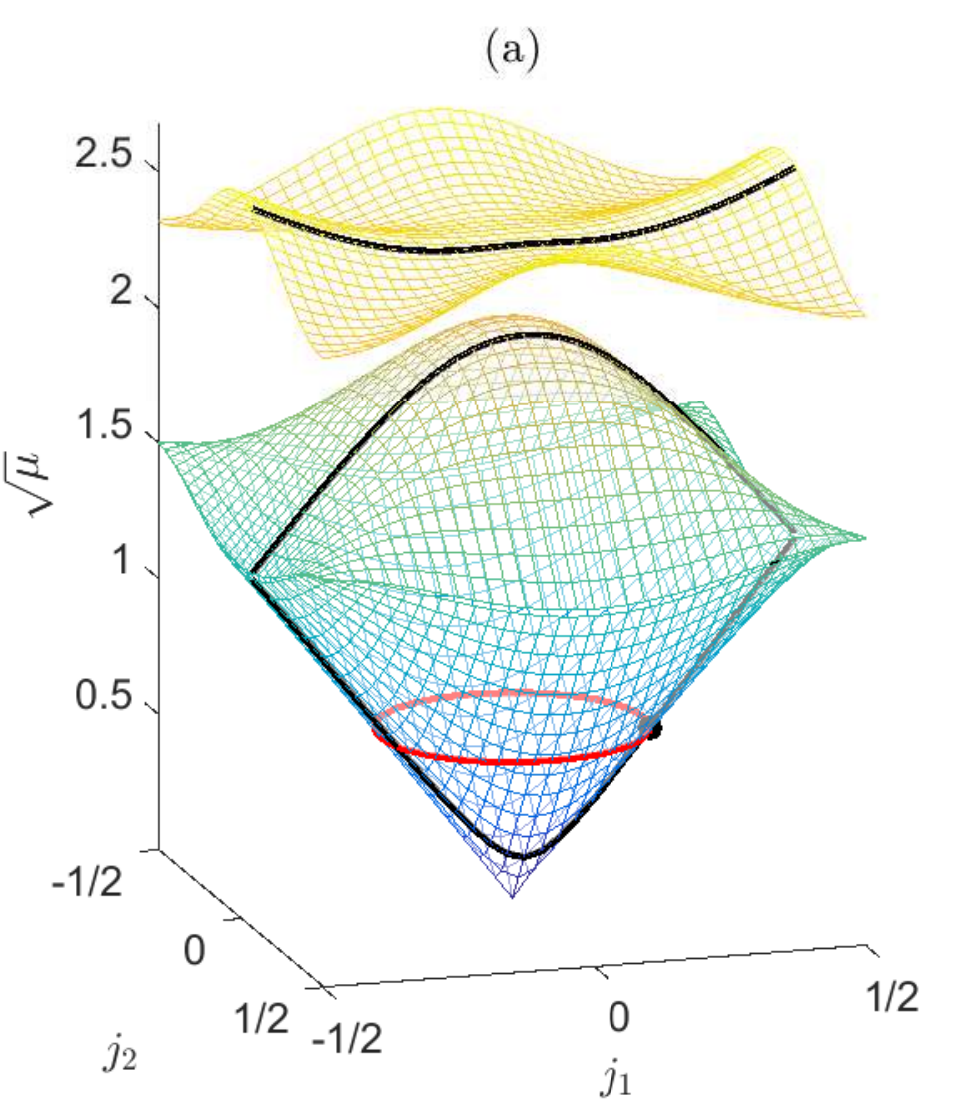}
    \includegraphics[scale=0.6]{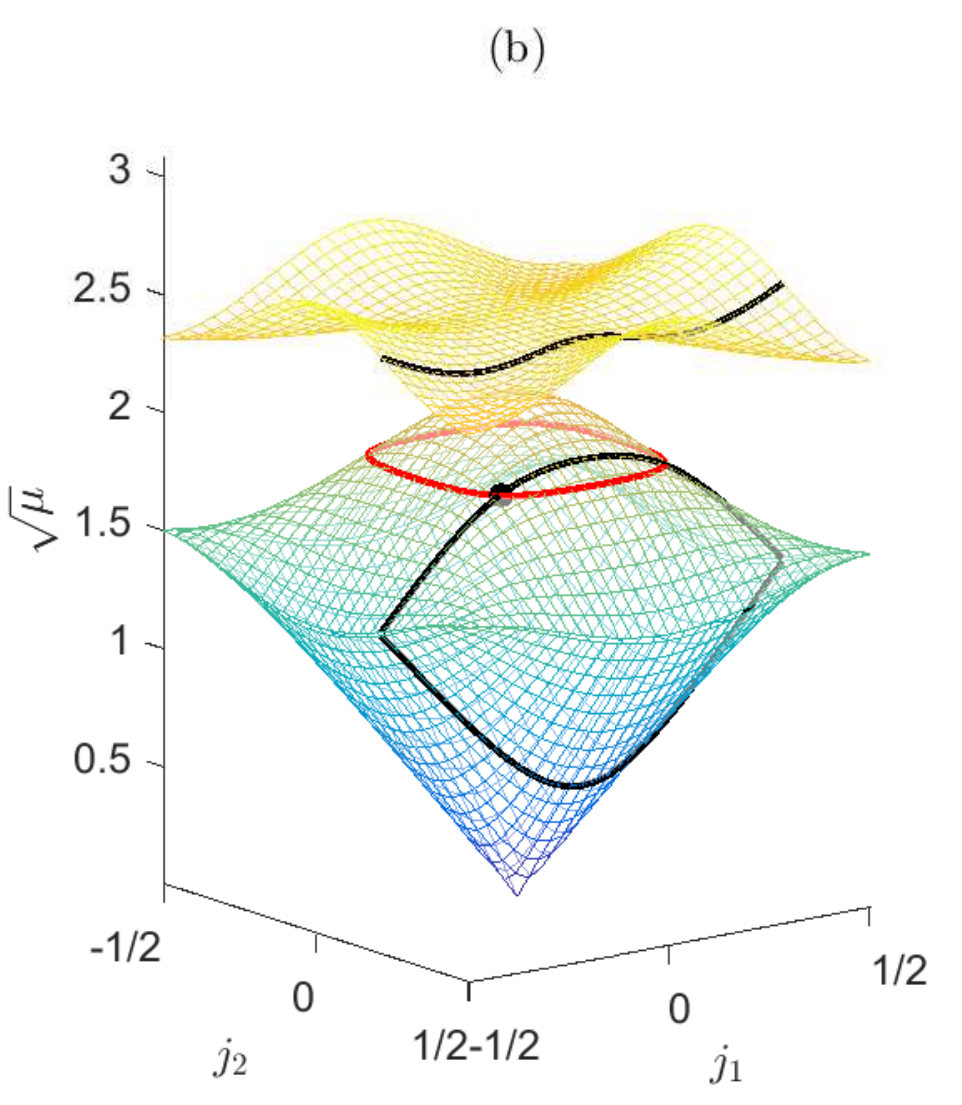}
    \caption{\label{F:comp_hom_bd_str} \small Band structure for $a_+$
      given in \eqref {eq:a+formula}. The three surfaces are the
      graphs of $\sqrt{\mu_m(j)}$ for $m=0,1,2$ (identical in (a) and
      (b)). The red line shows points on the graph that have height
      $\omega$. The black lines show points on the graph that satisfy
      $j_2 = j_2^{(\text{in})}$. The black dots are the
      $(j,\sqrt{\mu})$-coordinates of the ``transmission'' Bloch waves
      $U^{+,h}_{\lambda_i}$, $i=1,\dots,N_{\text{Bl}}$ with
      $\lambda_i\in I^{+,h}$ selected by the algorithm. (a) Situation
      for parameters $\omega$ and $j^{(\text{in})}$ of \eqref
      {E:incom-sm-om}. (b) Situation for the parameters of \eqref
      {E:incom-negref}. }
  \end{center}
\end{figure}
\begin{figure}[ht!]
  \begin{center}
    \includegraphics[scale=0.6]{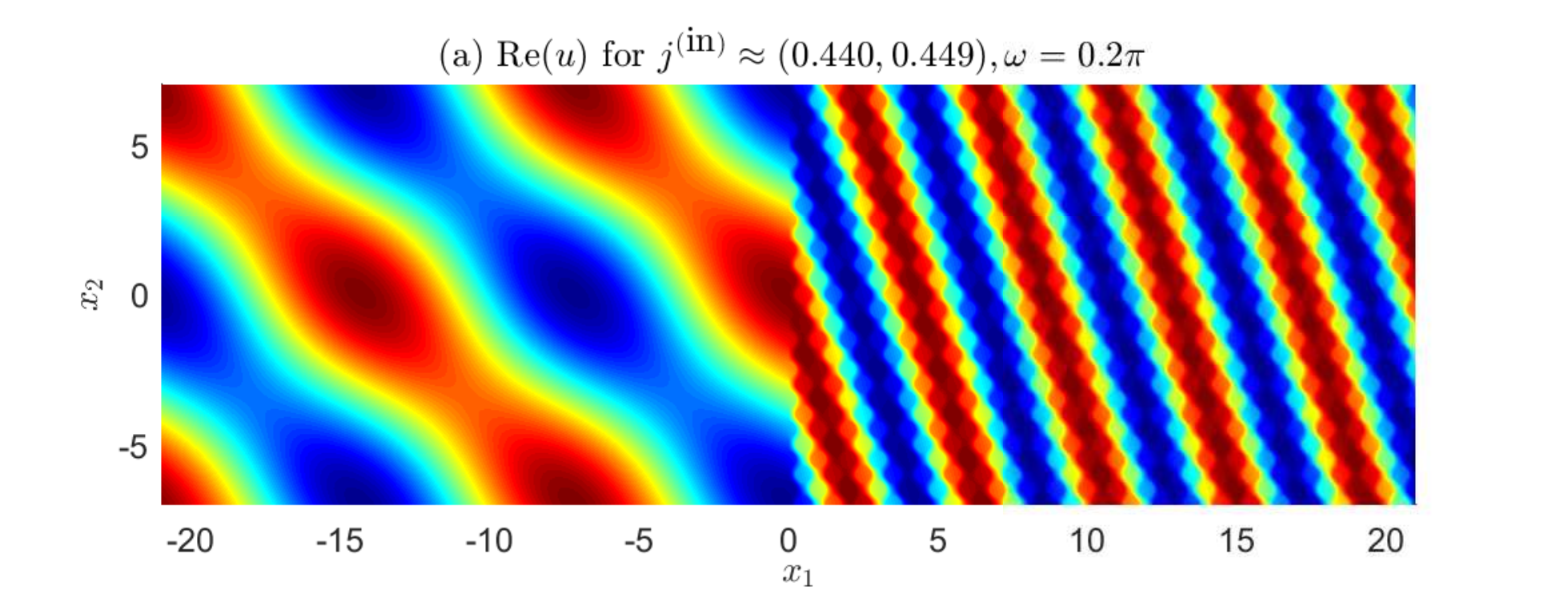}
    \includegraphics[scale=0.6]{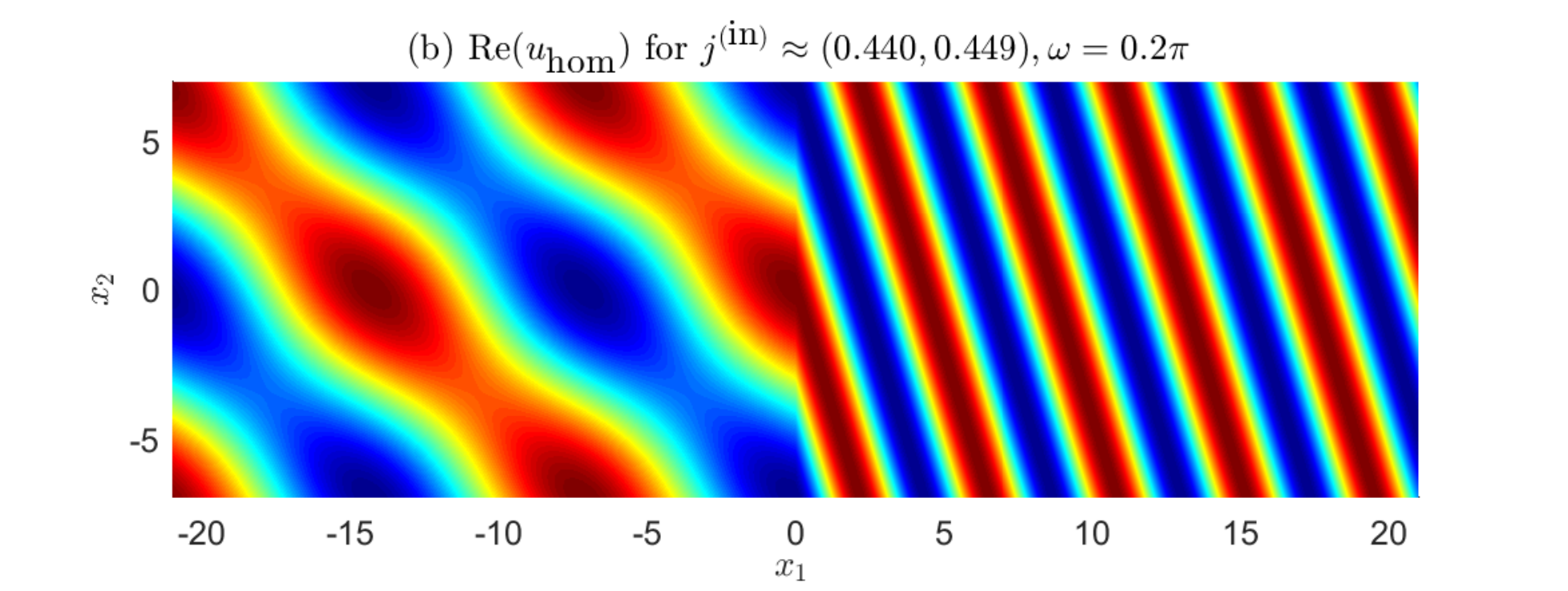}
    \caption{\label{F:comp_hom_om_sm} \small The color-coding shows
      $\text{Re}(u)$ on $\Omega_{R+L}$, where $u$ is the solution
      to an incoming wave given by \eqref{E:incom-sm-om}, the setting
      is that of Section \ref{S:comp_hom}. (a) A heterogeneous medium
      on the right as in \eqref{E:a-Luo}; (b) A homogeneous medium on
      the right with the homogenized coefficient $a_*\approx 0.1699$.}
  \end{center}
\end{figure}

In our second example we choose a frequency that is relatively large:
\begin{equation}\label{E:incom-negref}
  \omega =1.85, \ j^{(\text{in})}\approx (1.269, 1.346).
\end{equation}
The band structure at $\omega$ is far from a conical shape. Indeed,
the prediction of the homogenized model is not in agreement with the
solution for the heterogeneous material, see Figure
\ref{F:comp_hom_om_lg}.
\begin{figure}[ht!]
  \begin{center}
    \includegraphics[scale=0.6]{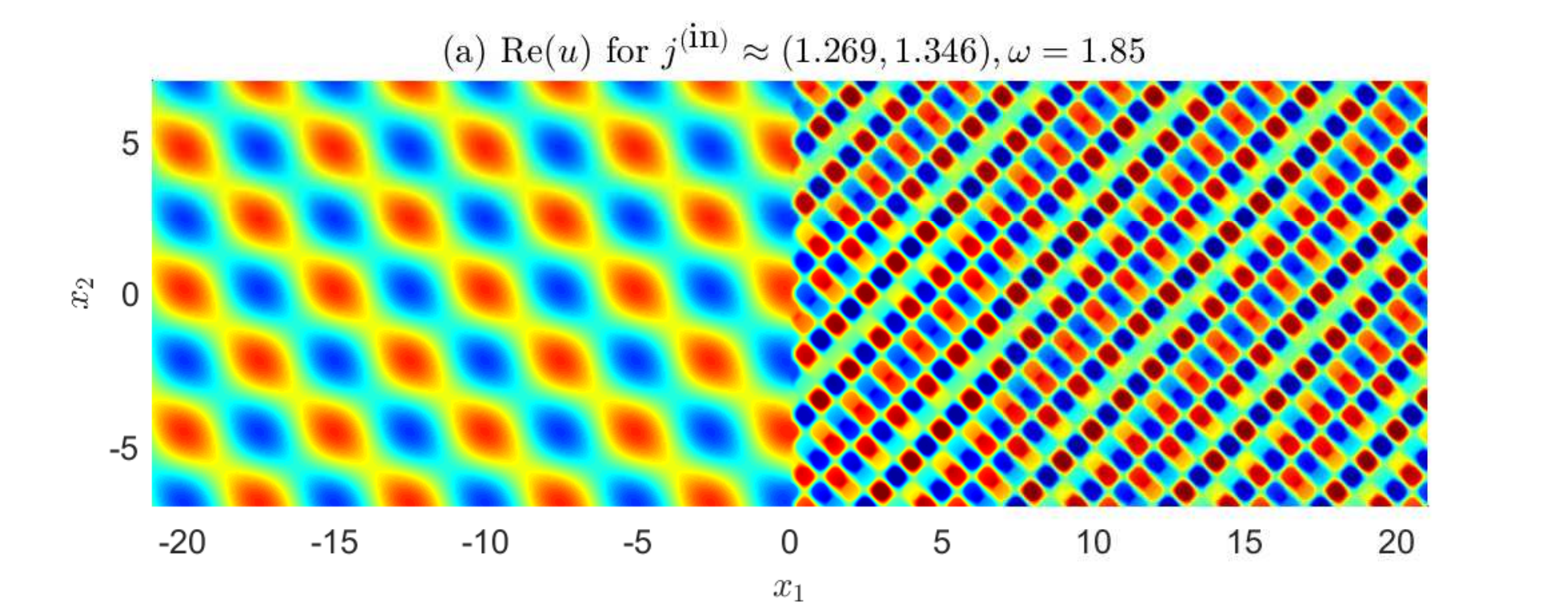}
    \includegraphics[scale=0.6]{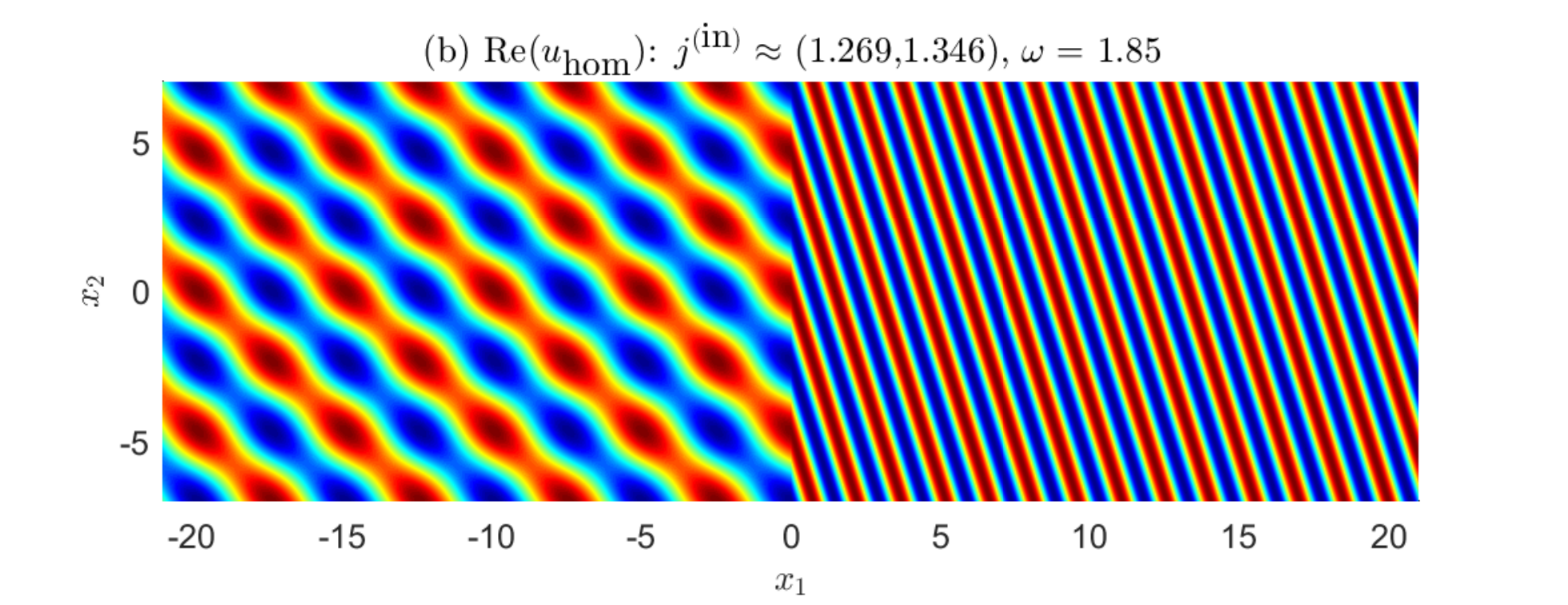}
    \caption{\label{F:comp_hom_om_lg} \small $\text{Re}(u)$ for the
      incoming wave given by \eqref{E:incom-negref} and the setting of
      Section \ref{S:comp_hom}. (a) The interface from
      \eqref{E:a-Luo}; (b) On the right half, $a_+$ is replaced by the
      homogenized coefficient $a_*\approx 0.1699$. One clearly sees
      negative refraction in (a), while homogenization predicts a
      positive refraction in (b).}
  \end{center}
\end{figure}

The parameters in \eqref{E:incom-negref} are chosen to produce
negative refraction. Negative refraction can be deduced from the
negative second component of the group velocity of the ``transmission"
Bloch waves $U^+_\lambda$, $\lambda \in I^+$ with frequency
$\omega$. In this situation, the incoming wave propagates upwards,
while the transmitted wave propagates downwards.  The group velocities
(multiplied by 2 for better visibility) of the ``transmission'' Bloch
waves $U^{+,h}_{\lambda_i}$, $i=1,\dots,N_{\text{Bl}}=4$ with
$\lambda_i\in I^{+,h}$ selected by the algorithm are plotted in Figure
\ref{F:k-vg_sel} (b). For completeness we show in Figure
\ref{F:k-vg_sel} (a) the group velocities of the ``reflected Bloch
waves'' $U^{-,h}_{\lambda_i}$. In both (a) and (b) all four group
velocity arrows lie very close to each other such that only one arrow
is visible.
\begin{figure}[ht!]
  \begin{center}
    \includegraphics[scale=0.6]{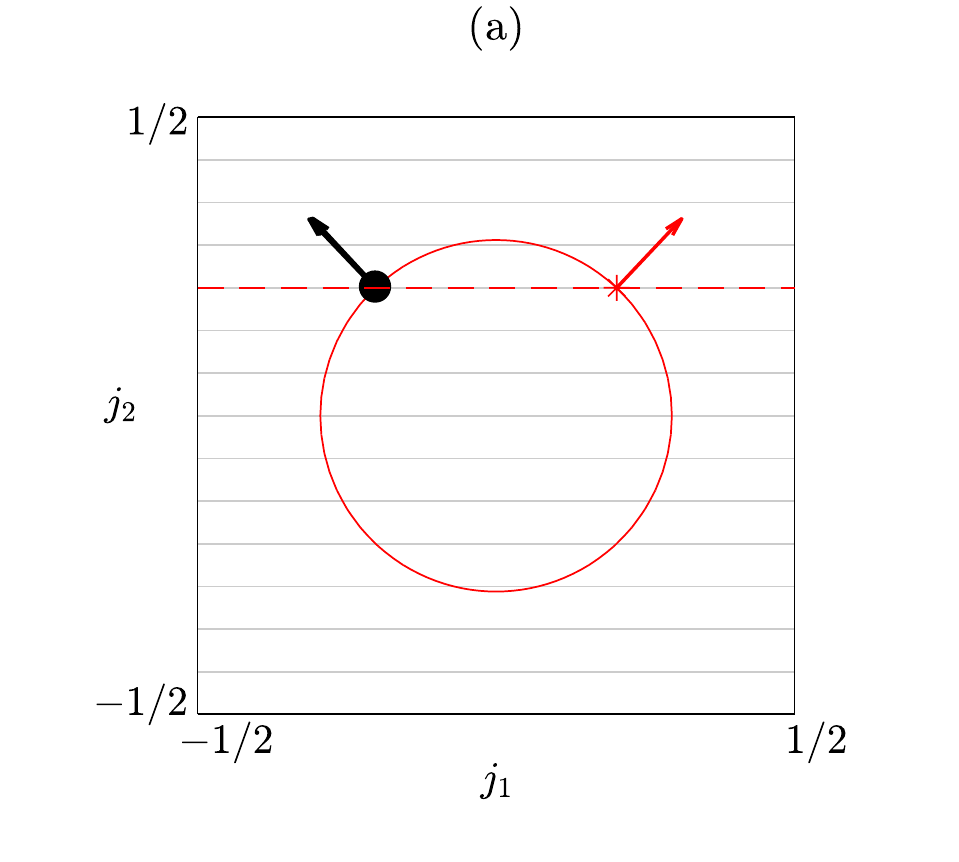}
    \includegraphics[scale=0.6]{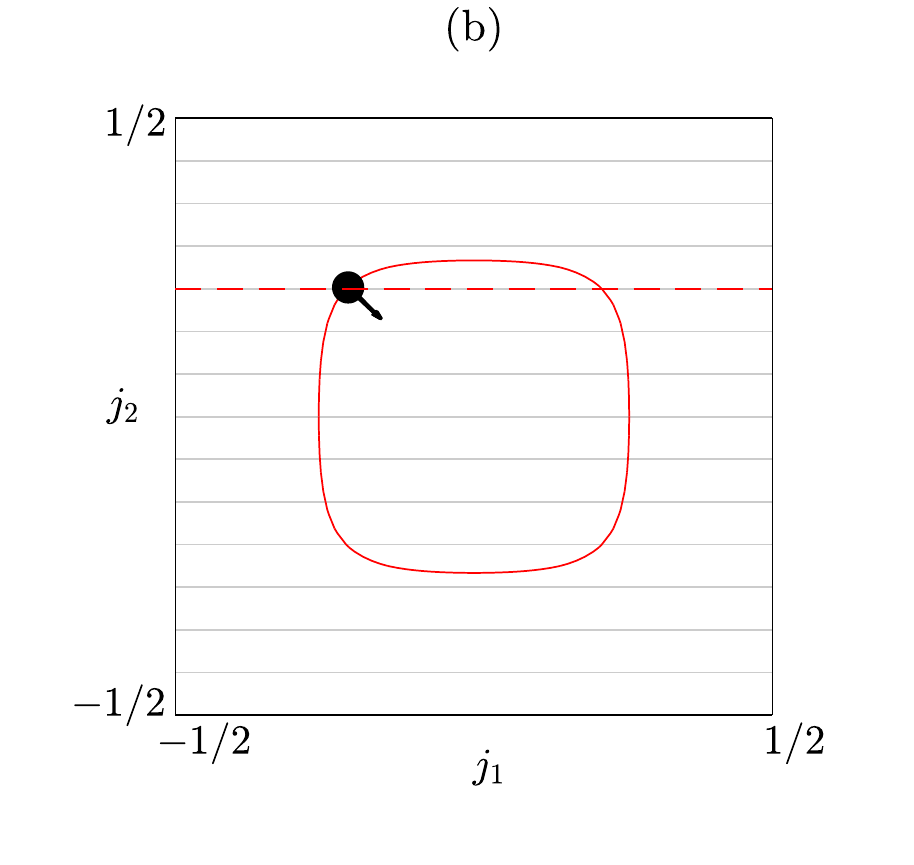}
    \caption{\label{F:k-vg_sel} \small The Brillouin zone
      $\mathbb{B}$. The closed red curves mark the level set of the
      band structure at the level $\omega^2$ given in
      \eqref{E:incom-negref} and with the setting of Section
      \ref{S:comp_hom}. Hence, the Bloch waves at the red points solve
      the Helmholtz equation with $\omega^2$.  The horizontal lines
      mark points that correspond to vertically periodic waves,
      $j_2\in Q'_K$. The arrows show the group velocities for those
      waves that are relevant in the numerical result.  The red arrow
      pointing north-east represents the incoming wave, the dashed
      horizontal line marks its $j_2$ component.  This component of
      $j_2$ is preserved across the interface.  (a) The situation in
      the left medium with the Bloch waves $U^{-,h}_{\lambda_i}$,
      $\lambda_i\in I^{-,h}$. (b) The situation in the right medium
      with the Bloch waves
      $U^{+,h}_{\lambda_i}, \lambda_i\in I^{+,h}$. In (b) the group
      velocities are multiplied by $2$ for better visibility. }
  \end{center}
\end{figure}

\paragraph{Refraction at interfaces between homogeneous media: Snell's
  law and Fresnel formulas.}

For a quantitative test of the numerical method we compute the
analytical solution for an interface separating two homogeneous
media. Here Snell's and Fresnel formulas are available and provide a
reference solution. We choose the same setting as in Figure
\ref{F:comp_hom_om_lg} (b).

For the interface with $a(x)=1$ for $x_1<0$ and $a(x)=a_*>0$ for
$x_1\geq 0$ Snell's law reads
$\sqrt{a_*} =\tfrac{\sin \theta_+}{\sin \theta_-}$, where
$\tfrac{1}{\sqrt{a_*}}$ is the refractive index of the material on
$x_1\geq 0$ and $\theta_+,\theta_-$ are the angles between
$j^{(\text{out})},j^{(\text{in})}$ and the horizontal axis,
respectively. Here $j^{(\text{out})}$ is the wavevector of the
transmitted wave.  For a straight vertical interface it is 
$j^{(\text{out})}_2=j^{(\text{in})}_2$, we therefore find
$\tfrac{\sin \theta_+}{\sin
  \theta_-}=\tfrac{|j^{(\text{in})}|}{|j^{(\text{out})}|}$. In
summary, Snell's law for this setting is
\begin{equation}\label{E:Snell}
  \sqrt{a_*}=\frac{|j^{(\text{in})}|}{|j^{(\text{out})}|}\,.
\end{equation}
Fresnel's formulas can be derived from the continuity of $u$ and
$\partial_{x_1}u$ across the interface. Writing
$$
u(x)=\begin{cases}
  e^{\ri j^{(\text{in})}\cdot x} + R e^{\ri (-j_1^{(\text{in})}x_1+j_2^{(\text{in})}x_2)}, \quad & x_1<0,\\
  Te^{\ri j^{(\text{out})}\cdot x}, \quad & x_1\geq 0\,,\\
\end{cases}
$$
we obtain
\begin{equation}\label{E:Fresnel}
  R=\frac{j_1^{(\text{in})}-a_*j_1^{(\text{out})}}{j_1^{(\text{in})}+a_*j_1^{(\text{out})}}\,,
  \quad\text{and}\quad
  T = 1+R\,.
\end{equation}
Given $j^{(\text{in})}$ and $a_*$, \eqref{E:Snell} and the equation
$j^{(\text{out})}_2=j^{(\text{in})}_2$ determine $j^{(\text{out})}$,
hence $R$ and $T$ can be evaluated from \eqref{E:Fresnel}. We compare
these values with the numerical ones. In the numerical results, we
interpret the wavevector of the dominant Bloch wave in $I^{+,h}$ as
the vector $j^{(\text{out})}$. We denote the correspoding index in
$I^{+,h}$ by $\lambda_\text{out}$.
Similarly, we denote by $\lambda_\text{refl}$ the index of the
dominant Bloch wave in $I^{-,h}$. The coefficients $R$ and $T$ are
approximated by the coefficient of the basis functions
$U^{-,h}_{\lambda_\text{refl}}$ and $U^{+,h}_{\lambda_\text{out}}$,
respectively (after renormalizing $U^{-,h}_{\lambda_\text{refl}}$ and
$U^{+,h}_{\lambda_\text{out}}$ such that
$\|U^{-,h}_{\lambda_\text{refl}}\|_{L^2([0,1]^2)}=\|U^{+,h}_{\lambda_\text{out}}\|_{L^2([0,1]^2)}=1$). We
denote these coefficients $\alpha_\text{refl}$ and
$\alpha_\text{out}$, respectively.

We use the incoming field as in Figure \ref{F:comp_hom_om_lg} (b),
i.e. that given in \eqref{E:incom-negref}. Discretizing with
$h_1=0.05$ and $h_2\approx 0.0526$, we get $a_*\approx 0.1699$
($\sqrt{a_*}\approx 0.4122$) and
$\frac{|j^{(\text{in})}|}{|j^{(\text{out})}|} \approx 0.414383$. This
approximation improves when the FEM-discretization is refined. In
Figure \ref{F:RT_conv} we plot the errors $|R-|\alpha_\text{refl}||$
and $|T-|\alpha_\text{out}||$ for the absorption parameter values
$\delta=10^{-p}, p=2,3,4,5,6$. For $\delta$ even smaller the errors do
not decrease due to the dominance of the discretization error. By
refining the discretization, the error for $\delta<10^{-6}$ can be
made smaller.
\begin{figure}[ht!]
  \begin{center}
    \includegraphics[scale=0.5]{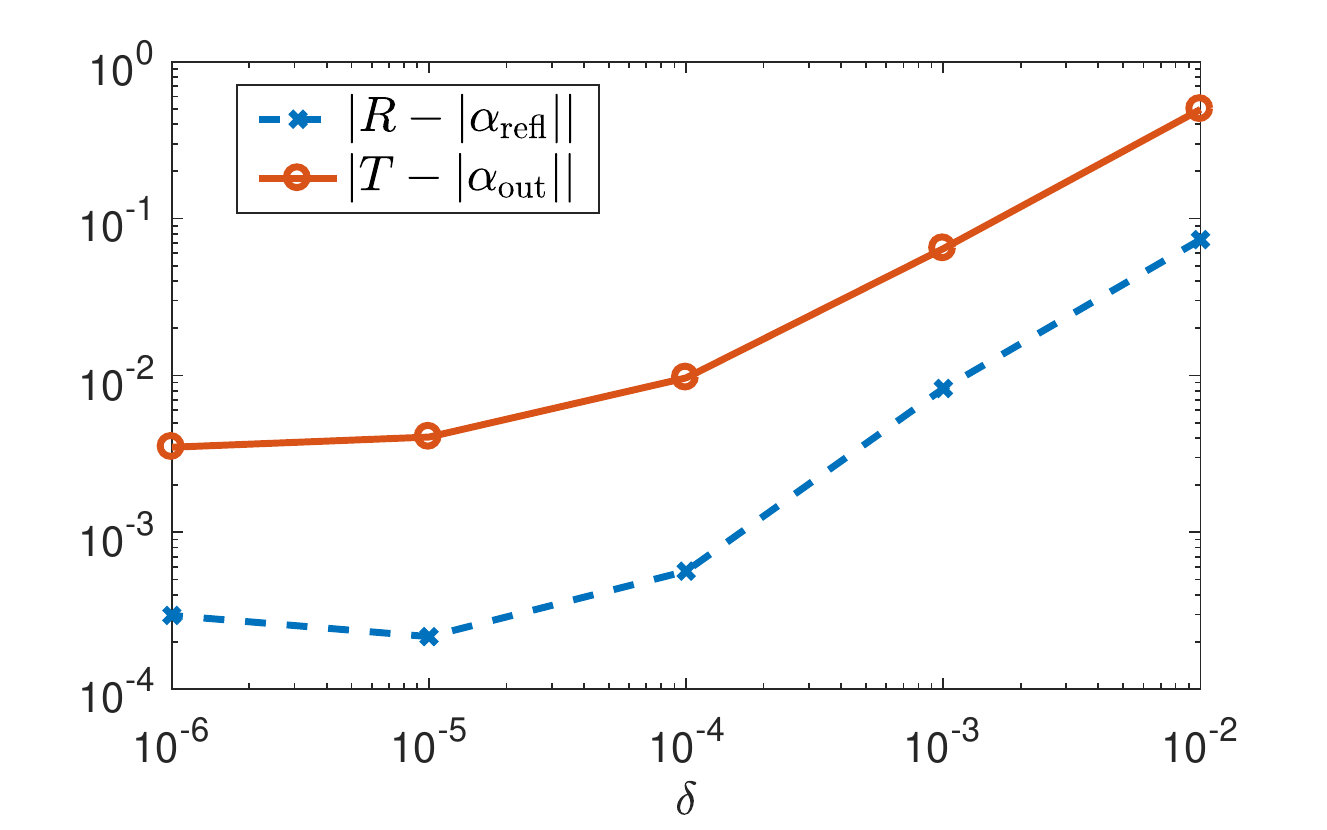}
    \caption{\label{F:RT_conv} \small Convergence of the error in the
      reflection and transmission coefficients $R$ and $T$ with
      respect to $\delta$.}
  \end{center}
\end{figure}

\subsection{Numerical results II: Scattering and negative refraction
  with a localized source}

In these tests we consider the same interface as in Section
\ref{S:comp_hom}. Instead of an incoming field we investigate a
spatially localized source; we choose
\begin{equation}\label{E:source}
f(x)=2e^{-3|x-x_*|^2}, \ x_*=(-3.5,0).
\end{equation}
This source generates waves in all directions, hence a relatively
large number $N_\text{Bl}$ of Bloch basis functions is needed.  We
consider a vertically wide domain in order for the periodic boundary
conditions to have a smaller effect near the source location. The
domain is given by $H=100$, $\eps=1$, $\eps R=45$, $\eps L =15$, and
we set $N_\text{Bl} = 180$. We choose again $\omega=1.85$. Figure
\ref{F:source} shows the solution modulus for $h_1 = h_2 =0.0625$. In
the crystal, close to the interface, the field is clearly focused in a
strip near the central line. We interpret that this effect is
generated by the negative refraction at the selected frequency.

In Figure \ref{F:source-ks} we plot once more the group velocities of
the Bloch waves $U^{-,h}_{\lambda_i}$, $\lambda_i \in I^{-,h}$ in (a)
and of $U^{+,h}_{\lambda_i}, \lambda_i \in I^{+,h}$ in (b). The size
of the dots at the foot of each arrow is proportional to the relative
modulus of the coefficient of the Bloch wave $U^{-,h}_{\lambda_i}$ in
(a) and $U^{+,h}_{\lambda_i}$ in (b). The strength of each Bloch wave
in the solution $u$ is thus visualized. Note that the large vertical
size $H$ leads to a large set $Q'_K$, which explains the large number
of gray horizontal lines in Figure \ref{F:source-ks}.
\begin{figure}[ht!]
  \begin{center}
    \def\big{\includegraphics[scale=0.75]{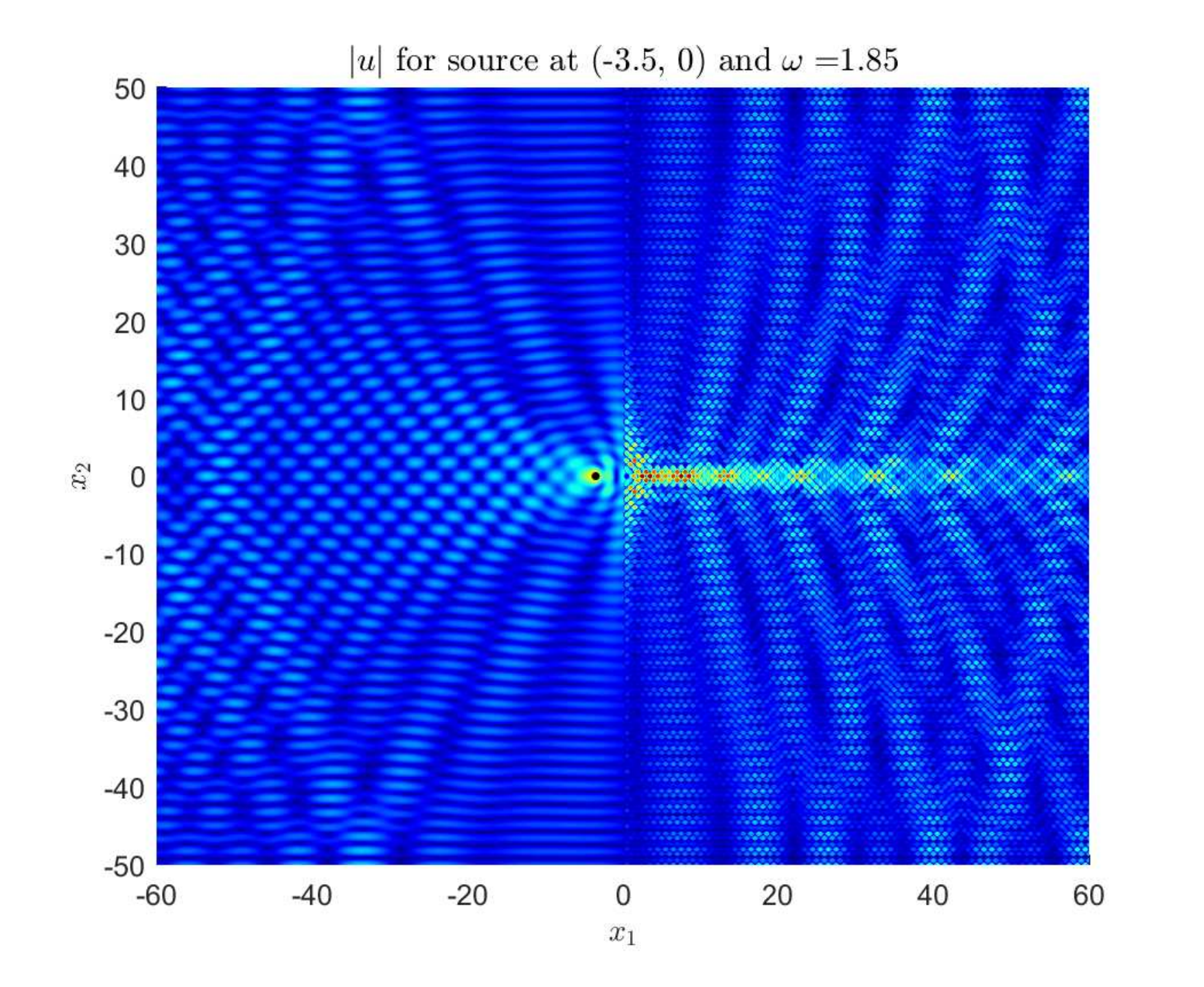}}
    \bottominset{\includegraphics[scale=0.4]{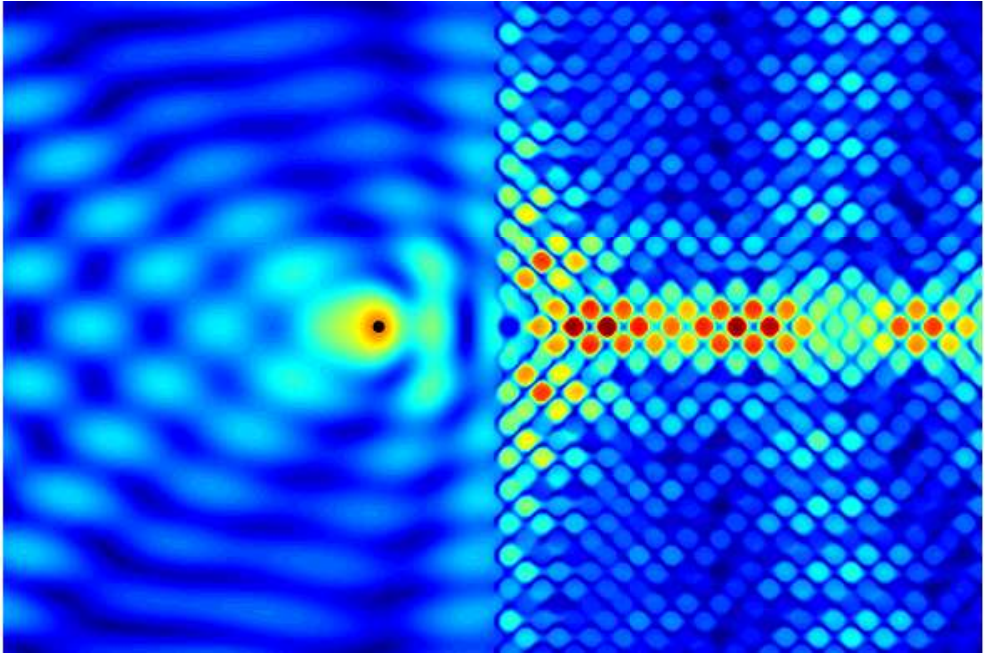}}{\big}{40pt}{-60pt}
    \caption{\label{F:source} \small $|u|$ for the scattering problem
      with the Gaussian source \eqref{E:source}. The interface is as
      in \eqref{E:a-Luo}, $\omega=1.85$ and $N_\text{Bl}=180$, $\eps
      R=45$, $\eps L =15$, $\eps=1$, $H=100$, $h_1=h_2=0.0625$. The
      inset zooms in on the vicinity of the center (near the source
      location).}
  \end{center}
\end{figure}

\begin{figure}[ht!]
  \begin{center}
    \includegraphics[scale=0.6]{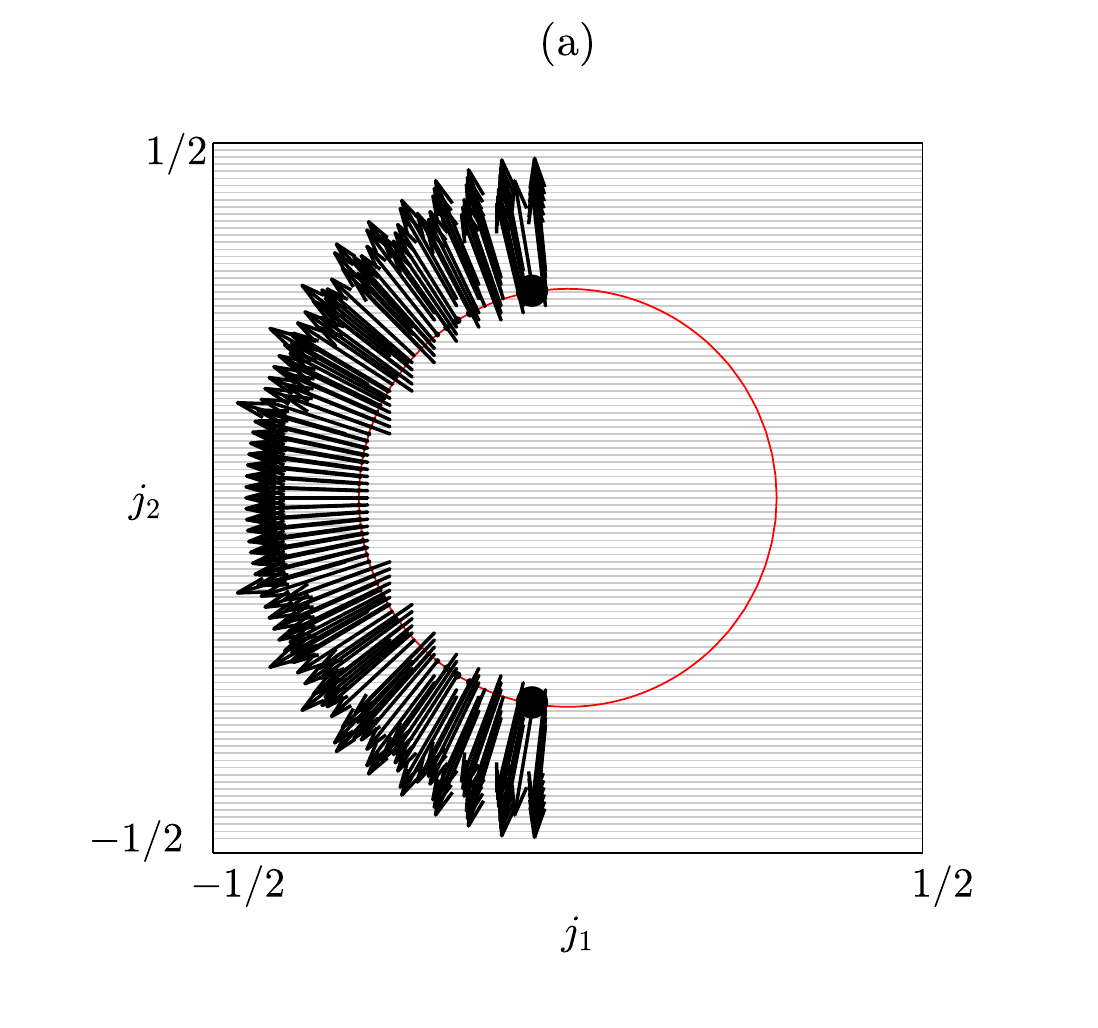}
    \includegraphics[scale=0.6]{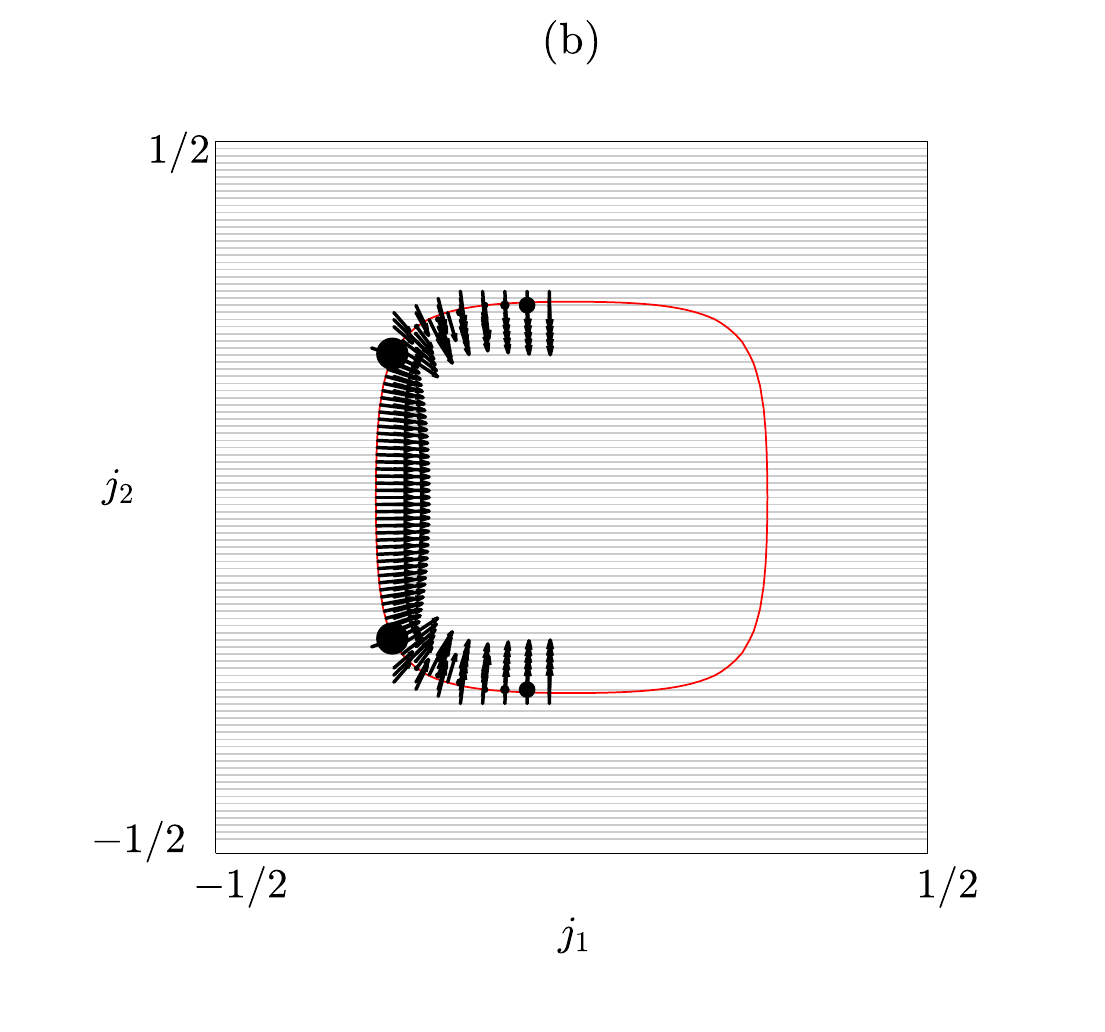}
    \caption{\label{F:source-ks} \small The Brillouin zone for the
      scattering problem of Figure \ref{F:source}, the symbols are as
      in Figure \ref{F:k-vg_sel}, (a) showing the left medium and (b)
      the right medium.  In contrast to the case of a single incoming
      wave, the numerical solution now uses many different Bloch
      waves, indicated by the arrows. Since all waves in $I^{\pm,h}$
      are outgoing, all arrows point to the left in the left medium
      and to the right in the right medium. The size of the dots at
      the foot of each arrow is porportional to the modulus of the
      coefficient of each Bloch wave.}
  \end{center}
\end{figure}

In order to confirm the lensing effect of a crystal for frequencies
with negative refraction, we use the same source as above but truncate
the crystal after 10 horizontal periods, see Figure
\ref{F:source_fin_crys}. With the frequency
$\omega =0.2\pi\approx 0.628$, where the refraction is positive, no
focusing occurs; contrastingly, at $\omega=1.85$, a focused image is
seen on the right side of the crystal. This confirms findings of
\cite{PhysRevB.65.201104}. We use here $\eps=1$, $\eps R=27$,
$\eps L =9$, $H=60$, $h_1\approx 0.0769$, $h_2\approx 0.0714$, and
$N_\text{Bl}=80.$
\begin{figure}[ht!]
  \begin{center}
	\includegraphics[scale=0.51]{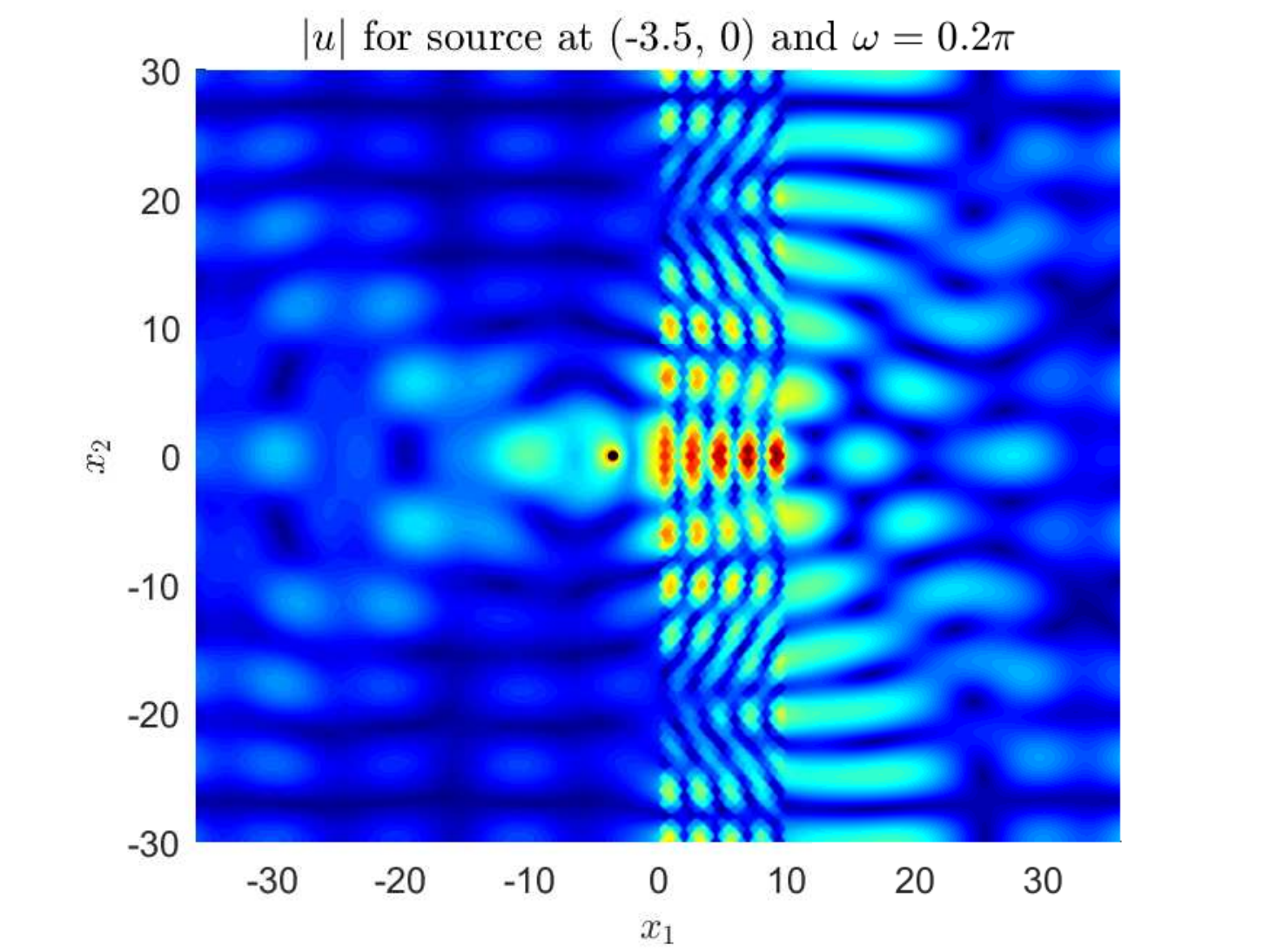}\hspace{-.6cm}
        \includegraphics[scale=0.51]{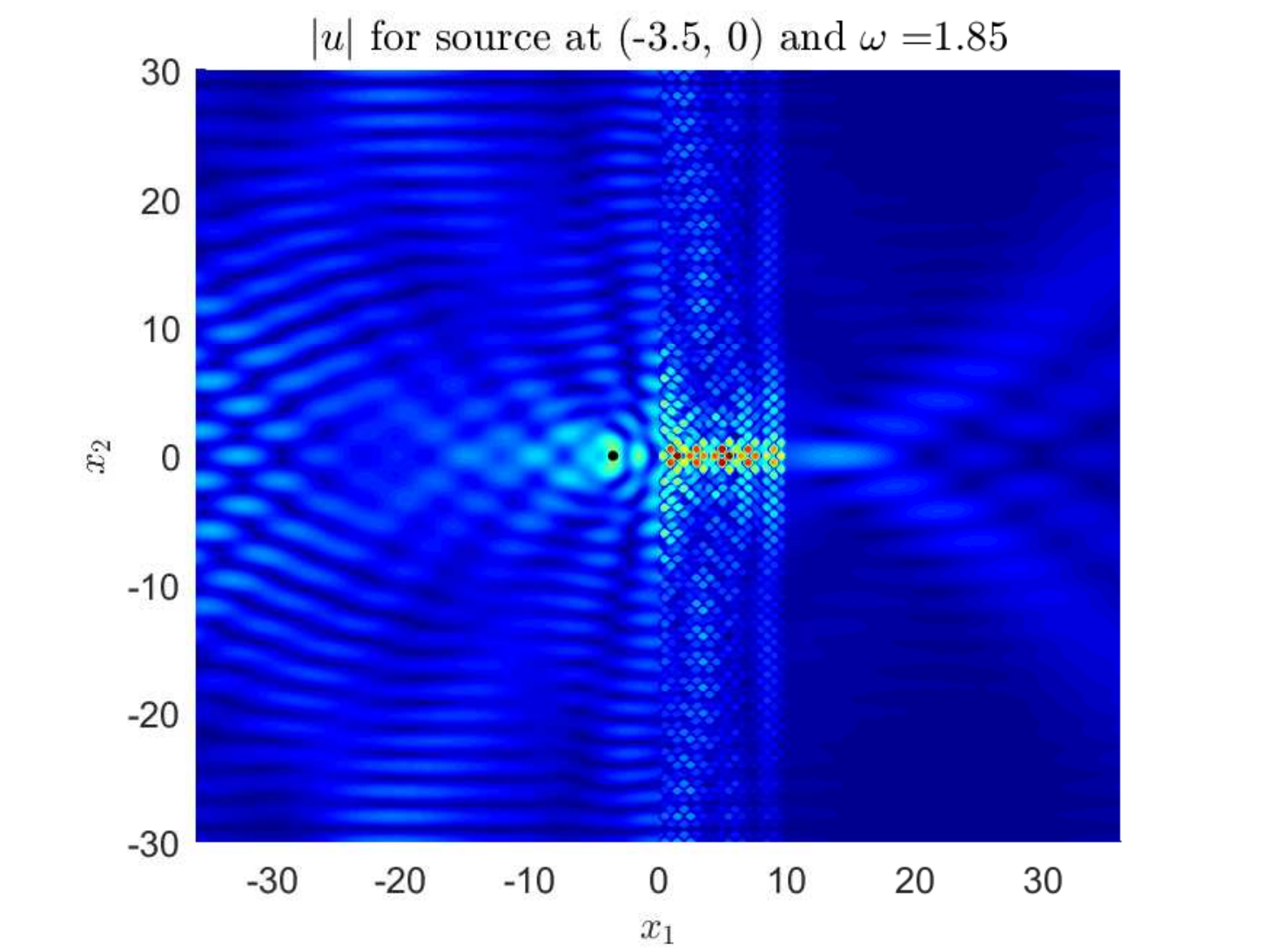}
        \caption{\label{F:source_fin_crys} \small $|u|$ for the
          scattering problem with the Gaussian source \eqref{E:source}
          and a crystal of thickness $10 \eps$, $\eps =1$. The
          material in $x_1\in [0,10\eps]$ is given by $a_+$ in \eqref
          {eq:a+formula} and by $1$ in $x_1\notin [0,10\eps]$. We use
          $N_\text{Bl}=80,$ $H=60$, $\eps R=27$, $\eps L =9$, and
          $h_1\approx 0.0769$, $h_2\approx 0.0714$. Left: Frequency
          $\omega=0.2\pi\approx 0.628$ leading to a positive
          refraction. Right: Frequency $\omega=1.85$ with a negative
          refraction and a resulting lensing effect.}
  \end{center}
\end{figure}

\subsection*{Acknowledgements}

Support of the first author by the DFG grant DO1467/3-1 and of the second author by the DFG grant Schw 639/6-1 is gratefully
acknowledged.

\bibliographystyle{abbrv} 
\bibliography{lit-Helmholtz-4}

\end{document}